\documentclass[11 pt]{amsart}

\usepackage[utf8]{inputenc}
\usepackage[margin=2.5cm]{geometry}
\usepackage{amsfonts,graphicx,enumerate}
\usepackage{amssymb, mathrsfs}
\usepackage{amsmath}
\usepackage{latexsym}
\usepackage{mathrsfs}
\usepackage{amsthm}
\usepackage{verbatim}
\usepackage{amscd}
\usepackage{mathtools}
\usepackage[pagebackref]{hyperref}
\hypersetup{bookmarksdepth=2}

\usepackage{import}
\usepackage{xifthen}
\usepackage{pdfpages}
\usepackage{transparent}
\usepackage{subcaption}

\usepackage{comment}
\usepackage{framed}
\usepackage{color}
\definecolor{shadecolor}{gray}{0.875}

\newtheorem{thrm}{Theorem}[section]
\newtheorem{lem}[thrm]{Lemma}
\newtheorem{cor}[thrm]{Corollary}
\newtheorem{prop}[thrm]{Proposition}
\newtheorem{conj}[thrm]{Conjecture}

\newtheorem*{thrma}{Theorem}

\theoremstyle{definition}
\newtheorem{defn}[thrm]{Definition}

\newtheorem{ex}[thrm]{Example}
\newtheorem{rmk}[thrm]{Remark}
\newtheorem{ques}[thrm]{Question}

\newtheorem*{rmka}{Remark}

\DeclareMathOperator{\st}{\bigm\vert}

\DeclareMathOperator{\Sym}{Sym}

\DeclareMathOperator{\Eff}{\overline{Eff}}

\DeclareMathOperator{\Nef}{{Nef}}

\DeclareMathOperator{\mum}{\overline{\mu}_{\rm min}}

\let\cal\mathcal
\let\frak\mathfrak
\let\bb\mathbb
\let\scr\mathscr

\usepackage[all]{xy}

\newcommand{\factor}[2]{\left. \raise 1pt\hbox{\ensuremath{#1}} \right/
        \hskip -2pt\raise -3pt\hbox{\ensuremath{#2}}}

\numberwithin{equation}{thrm}

\begin{document}

\title[New constructions of nef classes on self-products of curves]{New constructions of nef classes\\on self-products of curves}
\author{Mihai Fulger}
\address{Department of Mathematics, University of Connecticut, Storrs, CT 06269-1009, USA}
\address{Institute of Mathematics of the Romanian Academy, P. O. Box 1-764, RO-014700,
Bucharest, Romania}
\email{mihai.fulger@uconn.edu}
\thanks{The first author was partially supported by the Simons Foundation
Collaboration Grant 579353.}
\author{Takumi Murayama}
\address{Department of Mathematics\\Princeton University\\Princeton, NJ
08544-1000\\USA}
\email{takumim@math.princeton.edu}
\thanks{The second author was partially supported by the National Science
Foundation under Grant Nos.\ DMS-1701622 and DMS-1902616.}

\begin{abstract}
	We study the nef cone of self-products of a curve. When the curve is very general of genus $g>2$, we construct a nontrivial class of self-intersection 0 on the boundary of the nef cone. Up to symmetry, this is the only known nontrivial boundary example that exists for all $g > 2$. When the curve is general, we identify nef classes that improve on known examples for arbitrary curves. We also consider self-products of more than two copies of the curve. 
\end{abstract}

\maketitle


\section{Introduction}
The closure of the ample cone of a projective variety $X$ is the nef cone $\Nef(X)$.
It is a fundamental invariant that controls morphisms from $X$ to other projective varieties,
in particular projective embeddings of $X$. It is important to compute this cone in specific examples;
however, this is a difficult problem already on surfaces. 
Perhaps the most famous open question here is the following:

\begin{conj}[Nagata; see {\cite[Conjecture on p.\ 772]{NagataConj}}]\label{conj:nagata}Let $\pi\colon X\to\bb P^2$ be the blow-up of $n\geq 10$ very general points in $\bb P^2$ with exceptional divisors $E_1,\ldots,E_n$. Let $H\subset\bb P^2$ be any line. Then
	\[\pi^*(\sqrt nH)-E_1-\cdots-E_n\in\Nef(X).\]
\end{conj}

\noindent Recall that a property is very general on a variety if it holds outside a countable union of Zariski closed proper subsets. 
Conjecture \ref{conj:nagata} is a particular case of the SHGH conjecture. Note that the divisor in the Nagata conjecture has self-intersection 0, hence if it is nef, then it is on the boundary of the nef cone.

Another interesting class of surfaces is self-products of curves.
Recall the following open problem:

\begin{conj}[see {\cite[Remark 1.5.10]{laz04}}]\label{conj:prodcurvesintro}
  Let $C$ be a smooth projective curve of genus $g$ over $\bb C$.
  Denote by $f_1$ and $f_2$ (resp.\ $\delta$) the classes of the fibers of the
  projections (resp.\ the class of the diagonal $\Delta$) in $C \times C$.
  Then, we have
  \[
    (1+ \sqrt{g})(f_1+f_2) - \delta \in \operatorname{Nef}(C \times
    C)
  \]
  if $g$ is sufficiently large and $C$ has very general moduli.
\end{conj}

The self-intersection of $(1+\sqrt{g})(f_1+f_2) - \delta$ is $0$, just like in Conjecture \ref{conj:nagata}.
In fact, Ciliberto--Kouvidakis \cite{cknagata} and Ross \cite{Ross} prove that the Nagata conjecture implies Conjecture \ref{conj:prodcurvesintro}.
In the direction of Conjecture \ref{conj:prodcurvesintro}, 
Kouvidakis \cite[Theorem 2]{kouvidakis} shows that
\[\biggl(1+\frac g{\lfloor \sqrt g\rfloor}\biggr)(f_1+f_2)-\delta\in\Nef(C\times C).\]
In particular, the conjecture holds when $g$ is a perfect square.
An improvement when $g$ is not a perfect square is offered by \cite[(1.9)]{Ross} who uses work of \cite{SSS} to
prove that $\bigl(1+\sqrt{g+1}\bigr)(f_1+f_2)-\delta$ is nef.

It also makes sense to consider the non-symmetric divisors with zero self-intersection and ask:  

\begin{ques}\label{ques:main}
If $C$ is a very general curve of (large) genus $g$, and if $a>1$, is the class $a\:\!f_1+\bigl(1+\frac g{a-1}\bigr)f_2-\delta$ nef on $C\times C$?
\end{ques}


Just like with Conjecture \ref{conj:prodcurvesintro}, there are special curves for which the analogous question has a negative answer. For example, if $C$ is hyperelliptic and $a=2$, then $2f_1+(1+g)f_2-\delta$ is not nef.  
On arbitrary curves, the best known result is due to Rabindranath
\cite[Proposition 3.2]{ashwath}. 
He adapts an idea of Vojta \cite{vojta} to prove that
\begin{equation}\label{eq:vojtarabindranath}
  af_1+\biggl(1+\frac{g}{a-1}+(g-1)(a-1)\biggr)f_2-\delta\in\Nef(C\times C).
\end{equation}

\begin{figure}[t]
	\centering
	\includegraphics[scale=0.7]{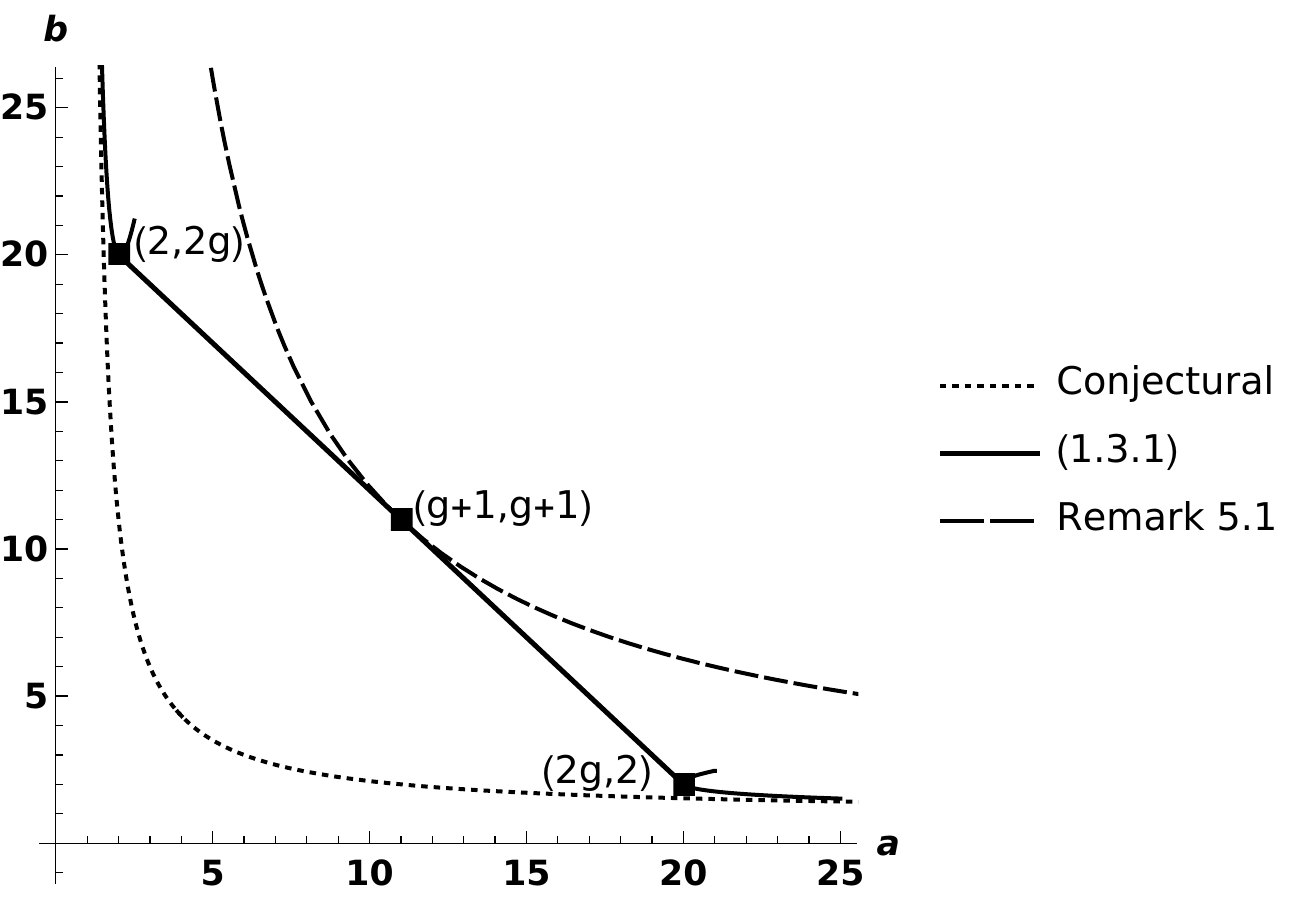}
	\caption{Nef classes on $C\times C$ for $g = 10$ and arbitrary $C$}
	{\footnotesize The classes $af_1+bf_2-\delta$ are represented by the points $(a,b)$. The outside curve is the conjectural nef boundary for very general curves in Question \ref{ques:main}. 
	Inside, on the left is the graph of $b=1+\frac g{a-1}+(a-1)(g-1)$ from \eqref{eq:vojtarabindranath}. On the right is its reflection. The dotted curve in the middle is $b=1+\frac{g^2}{a-1}$ from Remark \ref{rmk:delv}. The line segment bounds the convex hull. It is tangent to the curves at the specified points.}
	\label{fig:nefconeArbitraryg10}
\end{figure}

\noindent See Figure \ref{fig:nefconeArbitraryg10}. The line segment joining $(2,2g)$ and $(2g,2)$ is optimal for hyperelliptic curves.

\medskip

Our sharpest result answers Question \ref{ques:main} in the affirmative for
$a=2$. Up to symmetry, this is the only settled case of Question \ref{ques:main}
that we know of, other than Kouvidakis's when $g$ is a perfect square.

\begin{thrma}[see Theorem \ref{thrm:verygeneral2}]
	Let $C$ be a \emph{very general} smooth projective curve of genus $g\neq 2$.
	Then
	\begin{equation}\label{eq:verygenerala=2}
		2f_1+(1+g)f_2-\delta\in\Nef(C\times C).
	\end{equation}
\end{thrma}

The idea is to degenerate $C$ to a rational curve with $g$ simple nodes in general position using a construction of \cite{Ross}. The nefness of the limit of the classes \eqref{eq:verygenerala=2} follows from the elementary Proposition \ref{prop:p1xp1optimal} concerning the blow-up of $\bb P^1\times\bb P^1$ at $g$ general symmetric pairs of points.

In Corollary \ref{cor:genKouvidakis} we apply the original techniques of Kouvidakis. We degenerate to simple covers of $\bb P^1$ to show that 
for all integers $2\leq d\leq 1+\sqrt g$ and very general $C$:

\begin{equation}\label{eq:generalKouvidakispain}
	d\:\!f_1 + \biggl( 2 + \frac{2g}{d-1}-d \biggr) f_2 - \delta \in
	\operatorname{Nef}(C \times C).
\end{equation}

\noindent See Figure \ref{fig:nefconeVeryGeneralg10}. 

\begin{figure}
	\centering
	\includegraphics[scale=0.7]{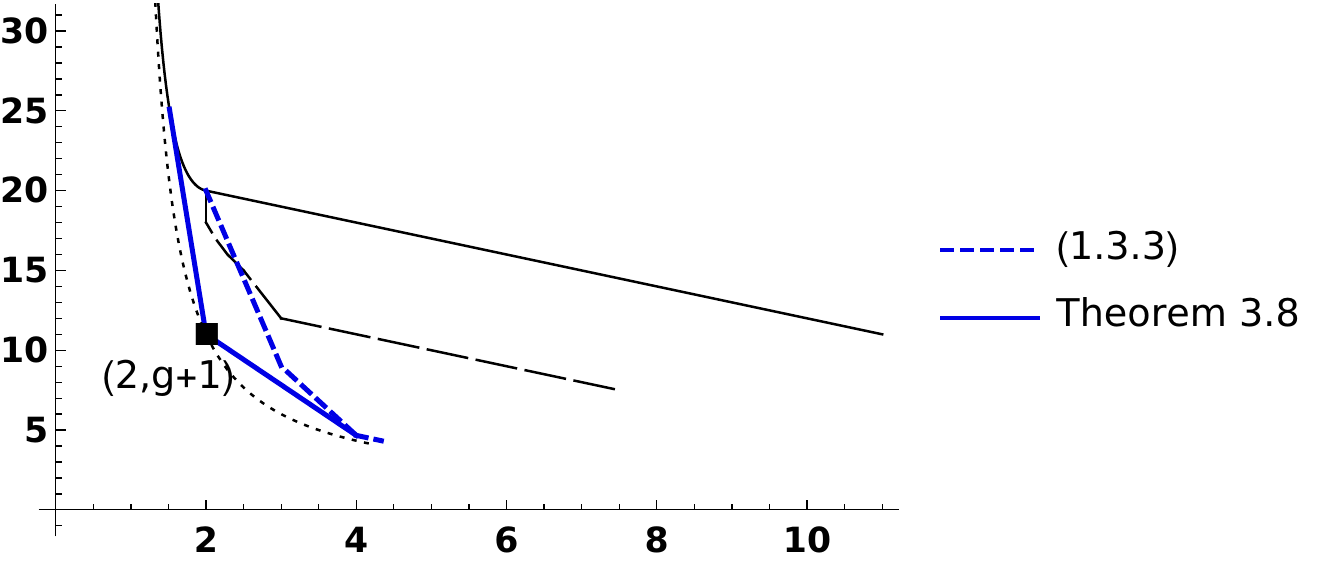}
	\caption{Nef classes on $C\times C$ for $g = 10$ and very general $C$}
	{\footnotesize To improve visibility, we only show the picture above the diagonal $a=b$. Note the difference in scale. The point $(2,g+1)$ comes from Theorem \ref{thrm:verygeneral2}. Keeping the contour from Figure \ref{fig:nefconeArbitraryg10} and from Figure \ref{fig:nefconeGeneralg10} below, we have added two polygonal lines. 
		The inner line represents classes coming from \eqref{eq:generalKouvidakispain}, and from convexity and symmetry. 
		In the outer polygonal line we show classes induced by Theorem \ref{thrm:verygeneral2} and by previous results.}
	\label{fig:nefconeVeryGeneralg10}
\end{figure}

\medskip

Next, we propose an approach to Question \ref{ques:main} in terms of semistability of vector bundles and give partial results.
We start with the simple observation that if $a>1$, then $af_1-\delta$ is ample on the fibers of the second projection $pr_2\colon C\times C\to C$. If $a$ is an integer and $L_a$ is a divisor of degree $a$ on $C$, then we observe in Proposition \ref{prop:curveapprox}.(\ref{prop:curveapproxharder}) that the ``positivity'' of $pr_1^*L_1-\Delta$ (in this case the smallest $b$ such that $af_1+bf_2-\delta$ is nef) is determined asymptotically by a similar measure of positivity of the sheaves $pr_{2*}\cal O\bigl(m(pr_1^*L_a-\Delta)\bigr)$. These are higher conormal bundles of $C$ in the sense of \cite{elstable}.
The idea is an instance of a general ``linearization'' principle that we explain in Proposition \ref{prop:positivitybypushforward}.

Fix a curve $C$ and a rational number $a>1$. In Theorem \ref{thrm:prodcurvesasymss}, we use the above to prove that Question \ref{ques:main} is true for $a$ and $C$ if and only if the higher conormal bundles above are semistable in an asymptotic sense on $C$.

Even in the asymptotic sense, understanding the semistability of the terms in
the sequence of higher conormal bundles seems out of reach.
However, the first (or 0-th, depending on convention) term $pr_{2*}\cal O(pr_1^*L_a-\Delta)$ is well-understood. It is the syzygy bundle $M_{L_a}$, i.e., the kernel of the evaluation morphism $H^0(C,L_a)\otimes\cal O_C\to\cal O_C(L_a)$. Drawing on known results about its semistability, we obtain the following:

\begin{thrma}[see Theorem \ref{thrm:prodcurvescgeneral}.(\ref{thrm:prodcurvescgenerali})]
  Let $C$ be a \emph{general} smooth projective curve of genus $g\geq 2$ over 
$\bb C$.
  Denote by $f_1$ and $f_2$ (resp.\ $\delta$) the classes of the fibers of the
  projections (resp.\ the class of the diagonal in $C \times C$).
  Then, we have
  \[
    d\:\!f_1 + \biggl( 1 + \frac{g}{d-g} \biggr) f_2 - \delta \in
    \operatorname{Nef}(C \times C)
  \]
  for every integer $d \ge \lfloor 3g/2\rfloor+1$.
\end{thrma}
When $d<2g$ and $g$ is large, we obtain examples outside the convex span of the known examples mentioned above due to Vojta and Rabindranath.
Slightly better bounds are given in Theorem \ref{thrm:prodcurvescgeneral}.(\ref{thrm:prodcurvescgeneralii}),(\ref{thrm:prodcurvescgeneraliii}).
See also Figure \ref{fig:nefconeGeneralg10}.

\begin{figure}[t]
	\centering
	\includegraphics[scale=0.7]{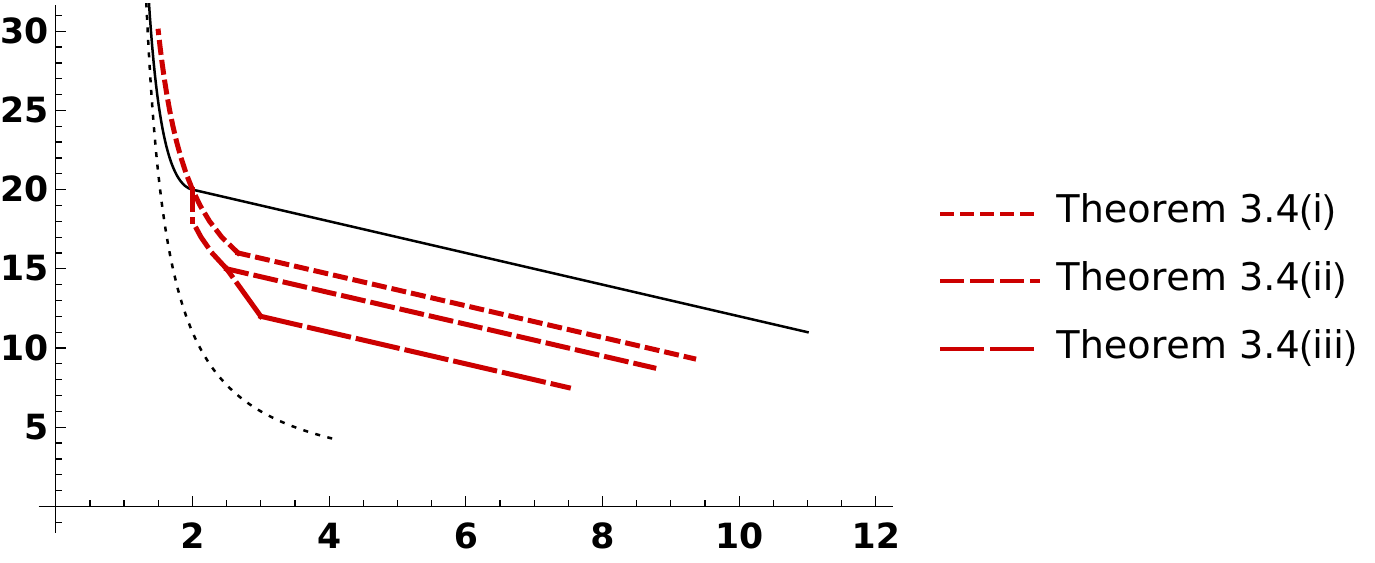}
	\caption{Nef classes on $C\times C$ for $g = 10$ and general $C$}
	{\footnotesize We again focus above the diagonal. Keeping the contour from the arbitrary case, we have added classes coming from the corresponding three parts of Theorem \ref{thrm:prodcurvescgeneral} and from the convexity and symmetry of the nef cone.
	} \label{fig:nefconeGeneralg10}
\end{figure}


\medskip

Finally, we also consider self-products of more than two copies of $C$.
Let $f_i$ be the class of any fiber of the $i$-th projection $C^n\to C$.
Let $\delta_{ij}$ be the class of the large diagonal $\{(x_1,\ldots,x_n)\st x_i=x_j\}$.
With assumptions as in Theorem \ref{thrm:prodcurvescgeneral}, it is immediate that
\[\sum_{i=2}^n\biggl(\Bigl(1+\frac{g}{d-g}\Bigr)f_1+df_i-\delta_{1i}\biggr)\in\Nef(C^n).\]
We show furthermore in Theorem \ref{thrm:mainCn} that if $C$ is an arbitrary smooth complex projective curve of positive genus and $d\in\bb Z$, then for certain values of $n$ and $d$,
\[(n-1)\cdot\biggl(1+\frac{g}{d-g}\biggr)f_1+d\cdot\sum_{i=2}^nf_i-\sum_{1\leq i<j\leq n}\delta_{ij}\in\Nef(C^n).\]
The proof makes use of the rich geometry of symmetric products of curves, and a result of Kempf on continuous global generation of vector bundles on abelian varieties.

\medskip

The material in this paper grew out of the work of the authors on Seshadri constants for vector bundles in \cite{fm19}. 
Versions of Theorems \ref{thrm:prodcurvescgeneral} and \ref{thrm:prodcurvesasymss} were initially included in its preprint form. 
They are now separated as they are of independent interest, 
and the proofs do not need the machinery of Seshadri constants. 
The results for very general curves and for products of arbitrarily many factors in Section
\ref{sec:higherproducts} did not appear in \cite{fm19}. 

\subsection*{Acknowledgments}
We thank Marian Aprodu, Thomas Bauer, Renzo Cavalieri, Alexandru Chirvasitu, Lawrence Ein, Mattias Jonsson,
Alex K\"uronya, Robert Lazarsfeld, Emanuele Macr\` i, Eyal Markman, Mircea
Musta\c{t}\u{a}, S\"onke Rollenske, Julius Ross, Praveen Kumar Roy, John Sheridan, and Brooke Ullery for useful discussions.

\section{Background and notation}

Let $X$ be a projective scheme over an algebraically closed field. While our main results are over $\bb C$,
some of our important tools (Proposition \ref{prop:curveapprox} and its generalization in Proposition \ref{prop:positivitybypushforward}) are valid in arbitrary characteristic. 

\subsection{Formal twists of coherent sheaves}
Let $\cal V$ be a coherent sheaf on $X$, and let $\lambda$ be an $\bb R$-Cartier $\bb R$-divisor on $X$.
Following the case of bundles in \cite[Section 6.2]{laz042},
the \emph{formal twist} of $\cal V$ by $\lambda$ is the pair $(\cal V,\lambda)$,
denoted by $\cal V\langle\lambda\rangle$.
When $D$ is an integral Cartier divisor, the formal twist $\cal V\langle D\rangle$ is identified with 
$\cal V\otimes\cal O_X(D)$. 

The theory of twisted \emph{vector bundles} has pullbacks. 
In particular, when $\cal V$ is a vector bundle and $D$ is a $\bb Q$-Cartier $\bb Q$-divisor and $f\colon X'\to X$ is a finite
morphism such that $f^*D$ is actually Cartier, then $f^*\cal V\langle
f^*D\rangle$ is $f^*\cal V\otimes\cal O_{X'}(f^*D)$.
 
If $\cal V$ is a vector bundle, we have Chern classes $c_1(\cal V\langle\lambda\rangle)\coloneqq c_1(\cal V)+\operatorname{rk}\cal V\cdot\lambda$. They are natural for pullbacks. 

Tensor products are defined by 
$\cal V\langle\lambda\rangle\otimes\cal V'\langle\lambda'\rangle\coloneqq  (\cal
V\otimes\cal V')\langle\lambda+\lambda'\rangle$.
Generally, when we talk about extensions, subsheaves, or quotients of twisted sheaves, or about morphisms between twisted sheaves, 
we understand that the twist $\lambda$ is fixed. 

Let $\bb P(\cal V)\coloneqq\operatorname{Proj}_{\cal O_X}\Sym^*\cal V$.
Let $\rho\colon\bb P(\cal V)\to X$ denote the natural projection map, 
and let $\xi$ denote the first Chern class of the relative 
$\cal O_{\bb P(\cal V)}(1)$ line bundle. 
If $\lambda$ is an $\bb R$-Cartier $\bb R$-divisor on $X$, 
define $\bb P(\cal V\langle\lambda\rangle)$ as $\bb P(\cal V)$, polarized with
the $\rho$-ample $\bb R$-Cartier $\bb R$-divisor
$\cal O_{\bb P(\cal V\langle\lambda\rangle)}(1)\coloneqq  \cal O_{\bb P(\cal V)}(1)\langle\rho^*\lambda\rangle$
whose first Chern class is $\xi+\rho^*\lambda$. 
This is in line with the classical formula $\cal O_{\bb P(\cal V\otimes\cal O_X(D))}(1)=\cal O_{\bb P(\cal V)}(1)\otimes\rho^*\cal O_X(D)$ whenever $D$ is a Cartier divisor. 

 The sheaf $\cal V$ is said to be \emph{ample} (resp.\ \emph{nef}) if the Cartier divisor class $\xi$ has the same property. This extends formally to twists.

\subsection{Slopes and positivity}
Assume that $C$ is a smooth projective curve. 
Let $\cal V$ be a coherent sheaf and let $\lambda$ be an $\bb R$-divisor. 
The \emph{degree} of $\cal V\langle\lambda\rangle$ is $\deg\cal V+\operatorname{rk}\cal V\cdot\deg\lambda$.
The \emph{slope} of the twisted coherent sheaf $\cal V\langle\lambda\rangle$ on $X$ is 
\[\mu(\cal V\langle\lambda\rangle)\coloneqq  \frac{\deg\cal V\langle\lambda\rangle}{\operatorname{rk}\cal V}.\]
By convention, the slope of torsion sheaves is infinite.
If $\cal V$ and $\cal V'$ are (twisted) coherent sheaves, then
\begin{equation}\label{eq:slope+}
\mu(\cal V\otimes\cal V')=\mu(\cal V)+\mu(\cal V').
\end{equation}
The smallest slope of any quotient of $\cal V$ is denoted by $\mu_{\rm min}(\cal V)$. 
Put $\mu_{\min}(\cal V\langle\lambda\rangle)\coloneqq\mu_{\min}(\cal V)+\deg\lambda$.
A quotient of $\cal V$ with minimal slope exists, and is determined by the Harder--Narasimhan filtration of $\cal V$.
In characteristic $0$, set $\overline{\mu}_{\rm min}(\cal V)\coloneqq  \mu_{\rm min}(\cal V)$.
In characteristic $p>0$, let $F\colon C\to C$ be the absolute Frobenius morphism, and consider
\[ \overline{\mu}_{\mathrm{min}}(\cal V) \coloneqq   \lim_{n \to \infty}
\frac{\mu_{\mathrm{min}}\bigl( (F^n)^*\cal V\bigr)}{p^n}.\]
The sequence in the limit is weakly decreasing and eventually stationary. 
In fact, \cite[Theorem 2.7]{Langer04} proves that there exists 
$\delta=\delta_{\cal V}\geq 0$ such that the Harder--Narasimhan filtration of $(F^{\delta+n})^*\cal V$ is 
the pullback of the Harder--Narasimhan filtration of $(F^{\delta})^*\cal
V$ for all $n\geq0$. 
In particular, the rational number $\overline{\mu}_{\rm min}(\cal V)=\frac{\mu_{\rm min}((F^{\delta})^*\cal V)}{p^{\delta}}$ 
is the smallest normalized slope of all quotients of all iterated Frobenius pullbacks $(F^n)^*\cal V$.
For twisted sheaves, put $\overline{\mu}_{\rm min}(\cal V\langle\lambda\rangle)=\mum(\cal V)+\deg\lambda$.

\begin{lem}[{\cite[Theorem 1.1]{BP14}}]\label{lem:bp14}
Let $C$ be a smooth projective curve over an algebraically closed field.
Let $\cal V\langle\lambda\rangle$ be a twisted vector bundle on $C$.
Denote $X\coloneqq\bb P(\cal V\langle\lambda\rangle)$ with bundle map
$\rho\colon X \to C$.
Denote by $\xi$ the numerical first Chern class of the relative (twisted) $\cal O_{\bb P(\cal V\langle\lambda\rangle)}(1)$ sheaf,
and by $f$ the class of a fiber of $\rho$. 
Then, we have
\[\Nef(X)=\bigl\langle \xi-\mum(\cal V\langle\lambda\rangle)\:\!f,f\bigr\rangle.\]
In particular, $\cal V\langle\lambda\rangle$ is nef if and only if $\mum(\cal V\langle\lambda\rangle)\geq 0$.
\end{lem}

The version in \cite{BP14} holds more generally for Grassmann bundles over curves.  
The result was seemingly first proved by Barton \cite[Theorem 2.1]{Barton71}. 
It is also stated explicitly by Brenner in \cite[Theorem 2.3]{Brenner04} and \cite[p.\ 534]{Brenner06}, 
Biswas in \cite[Theorem 1.1]{Biswas05}, and Zhao in \cite[Theorem 4.3]{Zhao17}.
In characteristic zero it follows easily from Hartshorne's Theorem \cite[Theorem 6.4.15]{laz042},
as observed by Miyaoka \cite{Miy87}. A similar computation is carried out by the first author in \cite{ful11} to nef classes of arbitrary codimension.

\begin{rmka}In positive characteristic, it is necessary to work with $\mum(\cal V)$ instead of $\mu_{\rm min}(\cal V)$. See \cite[Example 3.2]{har71} for a counterexample to the na\"ive positive characteristic analogue of Hartshorne's Theorem \cite[Theorem 6.4.15]{laz042}.
\end{rmka}

\section{Products of curves}
  Let $C$ be a smooth projective curve of genus $g$ over $\bb C$.
  Let $p$ and $q$ denote the projections onto each
  factor of $C\times C$.
  Let $f_1$ denote the class of the fiber of $p$ and $f_2$ the class of a fiber
  of $q$.
  Denote by $\delta$ the class of the diagonal $\Delta$.

For large genera, it is a tantalizing open problem to understand the nef cone of $C\times C$, 
even in the symmetric slice given by intersecting with the span of $f_1+f_2$ and $\delta$.

\subsection{Elementary considerations}
\begin{rmk}[Necessary conditions for nefness]
  Below, $a$, $b$, and $c$ denote non-negative real numbers.
  \begin{enumerate}[(1)]
    \item The classes $f_1$ and $f_2$ are clearly on the boundary of $\Nef(C\times C)$.
    \item The class $af_1+bf_2+c\delta$ is nef if and only if
      $(af_1+bf_2+c\delta)\cdot\delta=a+b-c(2g-2)\geq 0$.
      For example, $(g-1)f_1+(g-1)f_2+\delta$ is the pullback 
      of the theta polarization on the Jacobian of $C$ via the difference map
      \begin{align*}
        C \times C &\longrightarrow {\rm Jac}(C)\\
        (x,y) &\longmapsto \cal O_C(x-y).
      \end{align*}
       \item If $b$ and $c$ are not both zero, then the class $\pm af_1-bf_2-c\delta$ is not nef (or even
      pseudo-effective), because it has negative intersection with $f_1$. By symmetry, the analogous statement holds for $-af_1\pm bf_2-c\delta$ if $a$ and $c$ are not both zero.
\item The class $-af_1-bf_2+c\delta$ is only pseudo-effective when $a=b=0$, and only nef when $a=b=c=0$.
     \end{enumerate}
\end{rmk}

Thus, the classes that are not well-understood (up to scaling and
interchanging $f_1$ and $f_2$) are those of form 
\begin{enumerate}[(1)]
\item $af_1+bf_2-\delta$. By intersecting with $f_1$ and $f_2$, we get $a\geq 1$ and $b\geq 1$
as necessary conditions for these classes to be nef. By considering their self-intersections,
we also have $a>1$ and $b\geq 1+\frac g{a-1}$ as necessary conditions.
\item $-af_1+bf_2+\delta$. Here $0\leq a<1$ and $b\geq \frac g{1-a}-1$ are necessary conditions
for the class to be nef.
\item $af_1-bf_2+\delta$ with $0\leq b<1$ and $a\geq \frac g{1-b}-1$. 
\end{enumerate}

\begin{rmk}[Genus $g=1$]\label{rmk:g=1}
	The conditions above are also sufficient when $C$ is an elliptic curve. See \cite[Lemma 1.5.4]{laz04}.
\end{rmk}

Question \ref{ques:main}, which asks for the nefness of the class
\begin{equation}\label{eq:conjcurves}
  a\:\!f_1 + \biggl( 1 + \frac{g}{a-1} \biggr) f_2 - \delta,
\end{equation}
predicts that the conditions above are sufficient for classes of the form $af_1+bf_2-\delta$ for very general curves of sufficiently large genus.

\subsection{Nef classes for arbitrary curves}

\begin{rmk}[Rabindranath--Vojta divisors]\label{rmk:RabindranathVojta}
	Let $C$ be an \emph{arbitrary} smooth projective curve of genus $g\geq 1$.
Inspired by \cite{vojta}, Rabindranath in \cite[Proposition 3.2]{ashwath}
proves that if $r,s>0$, then 
$(\sqrt{(g+s)r^{-1}}+1)f_1+(\sqrt{(g+s)r}+1)f_2-\delta$
is nef if $r\geq \frac{(g+s)(g-1)}s$.\footnote{Note that \cite{ashwath} denotes our class $\delta-f_1-f_2$ by $\delta$.}
We thereby deduce the nefness of the divisor
\begin{equation}\label{eq:Vojta}
a\:\!f_1+\biggl(1+\frac g{a-1}+(g-1)(a-1)\biggr)f_2-\delta
\end{equation}
for $a>1$.
These are close to the conjectural bound \eqref{eq:conjcurves} for 
$a$ close to 1. See Figure \ref{fig:nefconeArbitraryg10}.

The original argument of Vojta \cite{vojta} applies to the classes $-af_1+bf_2+\delta$ with $0\leq a<1$ and proves that
\begin{equation}\label{eq:Vojta3}
-a\:\!f_1+\biggl(-1+\frac g{1-a}+(g-1)(1-a)\biggr)f_2+\delta\text{ is nef}.
\end{equation}
\end{rmk}

\subsection{Our main results for general curves}

We construct examples of nef classes for $C$ general. They improve
the examples in Remark \ref{rmk:RabindranathVojta} that were valid for arbitrary $C$.

\begin{thrm}\label{thrm:prodcurvescgeneral}
	
  Let $C$ be a \emph{general} smooth projective curve of genus $g$ over $\bb C$.
  Then
  
  \begin{enumerate}[\normalfont(i)]
  	\item\label{thrm:prodcurvescgenerali}  	
  	If $g\geq 2$, then for all integers $d\geq \lfloor 3g/2\rfloor +1$ the divisor class
  	\[
  	d\:\!f_1+\left(1+\frac g{d-g}\right)f_2-\delta\text{ is nef}.
  	\]
  	\item\label{thrm:prodcurvescgeneralii} 
  	  	If $g\geq 3$, then for all integers $2g-2\geq d\geq \lfloor 3g/2\rfloor$ the divisor class
  	\[
  	d\:\!f_1+\frac d{d-g+1}f_2-\delta\text{ is nef}.
  	\]
  	
  	\item\label{thrm:prodcurvescgeneraliii}
  	If $g\geq 10$, then the divisor class
  	\[
  	\bigl(\lfloor3g/2\rfloor-3\bigr)f_1+\frac{\lfloor3g/2\rfloor-3}{\lfloor g/2\rfloor-1}f_2-\delta\text{ is nef}.
  	\]
  	 \end{enumerate}
\end{thrm}

See Figure \ref{fig:nefconeGeneralg10} for a representation in genus 10
for how Theorem \ref{thrm:prodcurvescgeneral} improves Remark \ref{rmk:RabindranathVojta} for general curves.
When $C$ is an arbitrary smooth projective curve of genus $g\in\{0,1\}$, in fact $df_1+\bigl(1+\frac g{d-g}\bigr)f_2-\delta$ is nef for all \emph{real} $d>1$. See Remark \ref{rmk:g=1}.
The key ingredient of our proof of Theorem \ref{thrm:prodcurvescgeneral} is the semistability of kernel bundles.

\begin{defn}
	Let $X$ be a proper scheme over a field, and let $L$ be a locally free sheaf
  (usually a line bundle) on $X$. The \emph{kernel bundle} (sometimes called the
  \emph{Lazarsfeld--Mukai bundle}) $M_L$ is defined by the following exact sequence:
	\[0\longrightarrow M_L\longrightarrow H^0(X,L)\otimes\cal O_X\overset{{\rm ev}}{\longrightarrow} L.\]
\end{defn}

\noindent When $X=C$ is a curve, then $M_L=q_*(p^*L(-\Delta))$. In general $M_L=q_*(p^*L\otimes\cal I_{\Delta})$.
If $L$ is globally generated on $X$, then the invariants of $M_L$ are determined by $L$. For example $c_1(M_L)=-c_1(L)$ and ${\rm rk}\, M_L=h^0(X,L)-{\rm rk}\, L$.

\begin{proof}[Proof of Theorem \ref{thrm:prodcurvescgeneral}] 
	
	(\ref{thrm:prodcurvescgenerali}). On any smooth curve $C$, the bundle
	$M_{L}$ is semistable
	if $L$ is globally generated of degree $d$
	with $d-2(h^0(C,L)-1)\leq{\rm Cliff}(C)$.
	This result appears in several references, e.g., \cite{PR88}, \cite{Butler94}, \cite{Camere}, \cite{BrambilaOrtega},
	\cite[Theorem 1.3]{MisStop}, or \cite[Proposition 3.1]{EusenSchreyer}.
	
	When $C$ is general, ${\rm Cliff}(C)=\lfloor (g-1)/2\rfloor$ by \cite{acgh}. If furthermore $L$ is general of degree $d\geq\lfloor 3g/2\rfloor+1\geq g+2$, then $L$ is globally generated, and so is $L(-x)$ for general $x\in C$ (e.g., by \cite[\S3]{KM20}).
	Furthermore $L$ is non-special (i.e., $h^1(C,L)=0$), therefore $h^0(C,L)-1=d-g$, and $d-2(h^0(C,L)-1)=2g-d\leq\lfloor (g-1)/2\rfloor$.
	
	We deduce that $M_L$ is semistable of slope $-\frac d{d-g}=-\bigl(1+\frac g{d-g}\bigr)$. 
	By Lemma \ref{lem:bp14}, the twisted bundle $M_L\langle\frac{d}{d-g}x\rangle$ is nef.
	Furthermore the natural fiberwise evaluation map $\epsilon\colon q^*M_L\to p^*L(-\Delta)$ relative to $q$ specializes over general $x$ to $H^0(C,L(-x))\otimes\cal O_C\to L(-x)$, hence it is surjective on the general fiber. 
	Note that $p^*L(-\Delta)$ has degree $d-1$ on the fibers of $q$, hence it is relatively ample.
	Proposition \ref{prop:curveapprox}.(\ref{prop:curveapproxeasy}) applies to the twisted map $\epsilon\langle\frac{d}{d-g}q^*x\rangle$ proving that
	$p^*L(-\Delta)\langle\frac{d}{d-g}q^*x\rangle$ is nef. This proves the claim.
	
	(\ref{thrm:prodcurvescgeneralii}). When $L$ is globally generated of degree $d$ and $h^1(C,L)=1$, then $\mu(M_L)=-\frac d{d-g+1}$. 
	We look for $L$ satisfying the following properties:
	\begin{enumerate}[(a)]
		\item $L$ is globally generated of degree $d$ with $2g-2\geq d\geq \lfloor 3g/2\rfloor$.
		\item $h^1(C,L)=1$.
		\item $L(-x)$ is globally generated for general $x\in C$.
	\end{enumerate}
	Note that conditions (a) and (b) together imply $d-2(h^0(C,L)-1)\leq\lfloor(g-1)/2\rfloor$, which gives the semistability of $M_L$. In fact the inequality is strict.
	
	Condition (b) is equivalent by Serre duality to $L=\omega_C(-E)$ for some effective divisor $E$ with $h^0(C,E)=1$. Let $e=\deg E\geq 0$.
	Condition (b) is then met if we pick $E$ with $0\leq e<{\rm gon}(C)$. For general $C$, we have ${\rm gon}(C)=\bigl\lfloor\frac{g+3}2\bigr\rfloor=\bigl\lceil\frac g2+1\bigr\rceil$.
	
	We now focus on (c). Let $p\in C$ be an arbitrary point. The bundle $L(-x)$ is generated at $p$ if $H^1(C,L(-x-p))\to H^1(C,L(-x))$ is an isomorphism.
	By Serre duality, this is equivalent to the natural map $H^0(C,E+x)\to H^0(C,E+x+p)$ being an isomorphism, where $L=\omega_C(-E)$ as above. The map is in any case injective, and both spaces are at least 1-dimensional.
	It is sufficient to ask $h^0(C,E+x+p)=1$. As above, for $C$ general, this is implied by $2\leq e+2< \bigl\lfloor\frac{g+3}2\bigr\rfloor$.
	In fact, under this assumption, $L(-x)$ is globally generated for all $x$.
	
	Note that if $0\leq e\leq\bigl\lfloor\frac{g+3}2\bigr\rfloor-3$, then $d=2g-2-e$ is in the range $2g-2\geq d\geq\lfloor 3g/2\rfloor$. 
	
	Finally, to settle (a), we want to show that for all $0\leq e\leq \bigl\lfloor\frac{g+1}2\bigr\rfloor-3$ we can find $E$ effective of degree $e$ with $L=\omega_C(-E)$ globally generated. Arguing as in (c), we find that any effective $E$ will do. 
	
	(\ref{thrm:prodcurvescgeneraliii}). For $d=\lfloor\frac{3g}2\rfloor-3$,
	we have $e=2g-2-d=\lfloor\frac{g+3}2\rfloor$. 
	This is the gonality of a general curve. On such a general curve $C$ we can pick a divisor $E$
	of degree $e$ that is globally generated and $h^0(C,E)\geq 2$.
	In fact $h^0(C,E+x+p)=2$ for all $x,p\in C$ because $e+2=\lfloor\frac{g+3}2\rfloor+2<\frac {2g}3+2$,
	the next Brill--Noether threshold, under the assumption $g\geq 10$.
	As in part (\ref{thrm:prodcurvescgeneralii}) we deduce that $L=\omega_C(-E)$ and $L(-x)$ are globally generated for all $x\in C$.
	In this case $h^1(C,L)=h^0(C,E)=2$ and the inequality $d-2(h^0(C,L)-1)=e-2(h^0(C,E)-1)\leq \lfloor\frac{g-1}2\rfloor$ 
	that gives the semistability of $M_L$ is in fact an equality. 
	The divisors $L$ and $E$ compute the Clifford index of the curve.
	This is what restricts the result of (\ref{thrm:prodcurvescgeneraliii}) to just one class.		
\end{proof}

	
	

\begin{rmka}The Rabindranath--Vojta examples \eqref{eq:Vojta} or the generalized Kouvidakis classes \eqref{eq:genKouvidakis}  for $a=2$ 
give $2f_1+2gf_2-\delta\in\Nef(C\times C)$. 
Theorem \ref{thrm:prodcurvescgeneral}.(\ref{thrm:prodcurvescgenerali}) for $d=2g$ gives $2gf_1+2f_2-\delta\in\Nef(C\times C)$, which is the same class up to symmetry. 
For this reason, the divisors in Theorem \ref{thrm:prodcurvescgeneral}.(\ref{thrm:prodcurvescgenerali}) improve the Rabindranath--Vojta examples \eqref{eq:Vojta} only in the range $\lfloor 3g/2\rfloor+1\leq d<2g$, which is nonempty when $g\geq 3$. See Figure \ref{fig:nefconeGeneralg10}.

\end{rmka}

Towards answering Question
\ref{ques:main}, we expect that better bounds will arise
from an understanding of the semistability of higher-order
generalizations of the bundles $M_L$.
\begin{defn}\label{defn:higherconormal}If $L$ is a Cartier divisor on $C$ and $i\geq 0$ is an integer, we denote 
	\begin{align*}M^{(i-1)}(L)&\coloneqq q_*\bigl(p^*\cal O_C(L)\otimes\cal O_X(-i\Delta)\bigr)\\
		T^{i-1}(L)&\coloneqq q_*\bigl(p^*\cal O_C(L)\otimes\cal O_X(i\Delta)\bigr).
	\end{align*}
\end{defn}
\par We have $M^{(0)}(L)=M_L$.
When $L$ is globally generated, this is the pullback of the twisted cotangent space $\Omega\bb P^r(1)$ via the
morphism induced by $\lvert L\rvert$.
If $\lvert L\rvert$ is an embedding, then $M^{(1)}(L)=N^{\vee}_C\bb P^r\otimes\cal
O_C(L)$ is a twist of the conormal bundle. \cite{elstable}, Ein and
Lazarsfeld use the notation $R^{i-1}(L)$ instead of $M^{(i-1)}(L)$, and  call
them \emph{higher conormal bundles}.

\par We show that Question \ref{ques:main} can be restated in terms of 
the semistability in an asymptotic sense
of higher conormal bundles of $C$.

\begin{thrm}\label{thrm:prodcurvesasymss}
	Let $C$ be an \emph{arbitrary} smooth projective curve of genus $g$ over $\bb C$.
	\begin{enumerate}[\normalfont(i)]
		\item If $a>1$ is a rational number, then the class \eqref{eq:conjcurves}
		$af_1+\bigl(1+\frac g{a-1}\bigr)f_2-\delta$ is nef if and only if 
		the sheaves $M^{(n-1)}(nL)$ are \emph{asymptotically semi-stable}%
		\footnote{It makes sense to ask if $M^{(n-1)}(nL)$ is (semi)stable for large divisible $n$. See also \cite[Conjecture 4.2]{elstable}.}
		, i.e.,
		\[
		\lim_{n\to\infty}\frac 1n\mu_{\rm min}\bigl(M^{(n-1)}(nL)\bigr)=\lim_{n\to\infty}\frac 1n\mu\bigl(M^{(n-1)}(nL)\bigr),
		\]
		where $L$ is an arbitrary $\bb Q$-divisor on $C$ with $\deg L=a$, and $n$ is sufficiently divisible. 
		
		\item If $0\leq a<1$ is rational, then $-af_1+\bigl(-1+\frac g{1-a}\bigr)f_2+\delta$ is nef if and only if the sheaves $T^{n-1}(-nL)$ are asymptotically semi-stable.
	\end{enumerate}
	

\end{thrm}

\begin{proof}
	(i). Consider the $q$-ample class $af_1-\delta$. Since the class \eqref{eq:conjcurves}
	has self-intersection zero, it is nef if and only if 
	$\sup\bigl\{t\st af_1-tf_2-\delta\in\Nef(X)\bigr\}=-\bigl(1+\frac g{a-1}\bigr)$. 
	By Proposition
	\ref{prop:curveapprox}.(\ref{prop:curveapproxharder}), this holds if and only if
	$\frac{\mu_{\min}(M^{(n-1)}(nL))}n$ limits to $-1-\frac g{a-1}$.
	Recall that $L$ is a $\bb Q$-divisor on $C$ of degree $a$. 
	When computing the limit we restrict ourselves to $n$ such that $na\in\bb Z$.
	When $a>1$, for large divisible $n$ we have exact sequences 
	\[0\longrightarrow M^{(n-1)}(nL)\longrightarrow H^0\bigl(C,\cal O(nL)\bigr)\otimes\cal O_C\longrightarrow  P^{n-1}\cal O(nL)\longrightarrow
	0.\]
	Recall that if $\cal L$ is a line bundle, then $P^{n-1}\cal L$ denotes the bundle of principal parts $q_*(p^*\cal L\otimes\cal O_{n\Delta})$. 
	It is a rank $n$ vector bundle with a natural filtration with quotients 
	$\cal L$, $\cal L\otimes \omega_C$, \ldots, $\cal L\otimes \omega^{\otimes(n-1)}_C$. 
	From this, one computes
	\[
	\mu\bigl(M^{(n-1)}(nL)\bigr)=-n\biggl(1+\frac{ng}{na+1-g-n}\biggr).
	\]
	As $n$ grows, $\frac 1n\mu(M^{(n-1)}(nL))$ approaches $-\bigl(1+\frac{g}{a-1}\bigr)$.
	In particular, the nefness of $af_1+(1+\frac g{a-1})f_2-\delta$ is equivalent to the asymptotic semistability of $M^{(n-1)}(nL)$.
	\medskip
	
	(ii). Assume now $0\leq a<1$, and consider the $q$-ample class $-af_1+\delta$. 
	For large divisible $n$, pushing forward the exact sequence $0\to p^*\cal O(-nL)\to p^*\cal O(-nL)\otimes\cal O(n\Delta)\to p^*\cal O(-nL)\otimes\cal O(n\Delta)\vert_{n\Delta}\to 0$ by $q$, we obtain an exact sequence
	\[0\longrightarrow T^{n-1}(-nL)\longrightarrow q_*\bigl(p^*\cal O(-nL)\otimes\cal O(n\Delta)\vert_{n\Delta}\bigr)\longrightarrow H^1\bigl(C,\cal O(-nL)\bigr)\otimes\cal O_C\longrightarrow 0.\]
	From Riemann--Roch and by considering the surjections 
	\[q_*\bigl(p^*\cal O(-nL)\otimes\cal O(n\Delta)\vert_{(i+1)\Delta}\bigr)\twoheadrightarrow
	q_*\bigl(p^*\cal O(-nL)\otimes\cal O(n\Delta)\vert_{i\Delta}\bigr)\] whose kernels are isomorphic 
	to $\cal O(-nL)\otimes\omega_C^{\otimes(i-n)}$, one computes
	\[\frac 1n\mu\bigl(T^{n-1}(-nL)\bigr)=\frac{-n^2a-\binom{n+1}{2}(2g-2)}{n^2(1-a)+n(1-g)}.\]
	This limits to $-\bigl(\frac{g}{1-a}-1\bigr)$ as $n$ grows. 
\end{proof}

\begin{rmka}For $C$ an \emph{arbitrary} curve of genus $g$, \cite{EN18} prove that $M^{(k)}(L)$ is semi-stable if 
	$\deg L$ is exactly equal to $(k^2+2k+2)g+k$.
	This can be used to reprove the nefness of the divisors \eqref{eq:Vojta}
	when $b=1+\frac 1{k+1}$ with $k\geq 0$ an integer. 
\end{rmka}

\subsection{Nef classes for very general curves}
Our main result for very general curves constructs one optimal non-symmetric class, answering Question \ref{ques:main} in the affirmative for $a=2$.

\begin{thrm}\label{thrm:verygeneral2}
	Let $C$ be a very general smooth complex projective curve of genus $g\neq 2$. Then 
	\[2f_1+(1+g)f_2-\delta\in\Nef(C\times C).\]
\end{thrm}
\noindent When $g=2$, the class $2f_1+(1+g)f_2-\delta$ is not nef. It has negative intersection with the class $2f_1+2f_2-\delta$ of the graph of the hyperelliptic involution.

\begin{proof}
If $g=0$, then $2f_1+(1+g)f_2-\delta=f_1$ is nef (and not ample).
If $g=1$, then the class is on the boundary of the nef cone by Remark \ref{rmk:g=1}. We may assume then $g\geq 3$.
	The idea is to deform $C$ to a rational curve $C_0$ with $g$ simple nodes in general position. Since nefness is a very general condition in families, it is enough to prove a nefness statement for $C_0$.
The complication introduced by the nodes is that the positivity problem to be solved is on the blow-up of $\bb P^1\times\bb P^1$ at $2g$ points.
The construction comes from \cite{Ross}. We apply it to the non-symmetric situation.

Let $C_0$ be an irreducible rational curve with $g$ simple nodes in general position.
There exists a projective flat family $\cal C\to T$ over a disc $T$, relatively smooth with fibers of genus $g$ over the punctured disk, and with central fiber $C_0$. We may also assume that $\cal C$ has smooth total space. \footnote{$C_0$ is a stable nodal curve. Every stable curve of genus $g\geq 2$ embeds in $\bb P^{5g-6}$ with Hilbert polynomial depending only on $g$. 
	The set of points of the Hilbert scheme that correspond to stable embedded curves is denoted $H_g$.
	Let $Z_g\to H_g$ be the restriction of the universal family. \cite{DM69} prove that $H_g$ and $Z_g$ are smooth (even over ${\rm Spec}\bb Z$), and $H_g$ is irreducible over algebraically closed fields. Bertini arguments allow us to replace $Z_g\to H_g$ with $\cal C\to T$ as desired.} For $1\leq i\leq g$, denote by $x_i,y_i$ the preimages of each node in the normalization $\bb P^1$ of $C_0$. Let $L\subset \cal C$ be a section of $\cal C\to T$. It avoids the nodes of $C_0$. 

We would like to construct a Cartier divisor on $\cal C\times _T\cal C$ that restricts to the general fiber with class $2f_1+(1+g)f_2-\delta$. 
It is clear that for $f_1$ and $f_2$ we will use the pullbacks of $L$ by the two projections. However $\cal C\times_T\cal C$ is singular at the $g^2$ pairs of nodes, and the diagonal is not a Cartier divisor at the $g$ diagonal pairs. Instead we blow-up $\cal Y\to\cal C\times_T\cal C$ at the $g^2$ pairs of nodes. This resolves the singularities, in particular those along the diagonal. Let $\cal D\subset\cal Y$ be the strict transform of the diagonal.

For $t\neq 0$ in the disk $T$, the fiber $\cal Y_t=C_t\times C_t$ is the self-product of a genus $g$ curve. For $t=0$, the fiber $\cal Y_0$ has $g^2$ exceptional $\bb P^1\times\bb P^1$ components, and a component $F$, the strict transform of $C_0\times C_0$. Let $\nu\colon\widetilde F\to F$ be the normalization. As a variety, $\widetilde F$ is isomorphic to the blow-up of $\bb P^1\times\bb P^1$ at the $4g^2$ ordered pairs of points from the list $\{x_1,y_1,\ldots,x_g,y_g\}\subset\bb P^1$. See \cite[Lemma 3.1]{Ross} for the proofs. Denote the classes of the exceptional $\bb P^1$'s over the corresponding points by $e_{x_ix_j},e_{x_iy_j},e_{y_ix_j},e_{y_iy_j}$. Let $\pi$ be the blow-up of $\bb P^1\times\bb P^1$. 

Let $E$ be the sum of the $g$ exceptional $\bb P^1\times\bb P^1$'s sitting over diagonal pairs of nodes $(p,p)$. Over each of the $g$ components of $E$, the divisor $\cal D$ restricts with class $f_1$ in $N^1(\bb P^1\times\bb P^1)$, while $E$ restricts with class $-2f_1-2f_2$.
By \cite[Lemma 3.2]{Ross}, we have $\nu^*(\cal D|_{F})=\pi^*\Delta_{\bb P^1}-\sum_{i=1}^g(e_{x_ix_i}+e_{y_iy_i})$. Furthermore $\nu^*( E|_{F})=\sum_{i=1}^g(e_{x_ix_i}+e_{y_iy_i}+e_{x_iy_i}+e_{y_ix_i})$.
With $p$ and $q$ denoting the induced projections from $\cal Y$ on the factors of $\cal C\times_T\cal C$, consider on $\cal Y$ the Cartier divisor \[N\coloneqq p^*2L+q^*(1+g)L-(\cal D+E).\]
If we prove that its restriction $N|_{\cal Y_0}$ is nef, then the same holds
for the restriction to the very general fiber. Clearly for $t\neq 0$ the fiber restriction has class $2f_1+(1+g)f_2-\delta\in N^1(C_t\times C_t)$.

The restriction of $N$ to the exceptional $\bb P^1\times\bb P^1$ components has class $f_1+2f_2$, so it is even ample. On the other hand, the class of $\nu^*(N|_F)$ is 
\begin{align*}
  \pi^*\bigl(2f_1+(1+g)f_2\bigr)-\biggl(\pi^*\delta&-\sum_{i=1}^g(e_{x_ix_i}+e_{y_iy_i})\biggr)-\sum_{i=1}^g(e_{x_ix_i}+e_{y_iy_i}+e_{x_iy_i}+e_{y_ix_i})\\
  =\pi^*(f_1+gf_2)&-\sum_{i=1}^g(e_{x_iy_i}+e_{y_ix_i})\in N^1(\widetilde F).
\end{align*}
To settle the nefness of this class, it was enough to blow-up only the $2g$ points $(x_i,y_i)$ and $(y_i,x_i)$ with $1\leq i\leq g$ on $\bb P^1\times\bb P^1$.
The conclusion follows from the result below.
\end{proof}

\begin{prop}\label{prop:p1xp1optimal}
	Consider general points $z_1,\ldots,z_g\in\bb P^1\times\bb P^1$ with $g\neq 2$.
	Let $\pi\colon X\to\bb P^1\times\bb P^1$ be the blow-up of the $2g$ points $z_1,\ldots,z_g$ and their reflections $z_1',\ldots,z_g'$ across the diagonal. Denote by $E$ the exceptional divisor. Then \[\pi^*(f_1+gf_2)-E\in\Nef(X).\] 
	For all $g\geq 0$, the same nefness result holds if we blow-up $2g$ general points in $\bb P^1\times\bb P^1$.
\end{prop}

\begin{proof}	
	\texttt{Step 1. The case $g\in\{0,1\}$}.
	The case $g=0$ is trivial. When $g=1$, then $\pi^*(f_1+f_2)-E$ is represented by $\overline F_1+\overline F_2$, where $\overline F_1$ is the strict transform of the fiber of the first projection through $z_1$, and $\overline F_2$ is the strict transform of the fiber of the second projection through $z_1'$. We have $\overline F_1^2=\overline F_2^2=-1$. We may assume that $z_1$ is not on the diagonal, hence $\overline F_1\cdot \overline F_2=1$. In particular $\overline F_1+\overline F_2$ has nonnegative (in fact 0) intersection with each of its irreducible components, hence it is nef.

	\texttt{Step 2. The failure of the case $g=2$.} $\bb P^2$ can be identified with the second symmetric power of $\bb P^1$. The sum map $\sigma\colon\bb P^1\times\bb P^1\to\bb P^2$ given by $\sigma(x,y)=x+y$ is a cover of degree 2 and $\sigma^*\cal O_{\bb P^2}(1)=\cal O(1,1)$. 
	The line through $\sigma(z_1)$ and $\sigma(z_2)$ lifts to a plane section of $\bb P^1\times\bb P^1$ through $z_1,z_1',z_2,z_2'$ in the Segre embedding in $\bb P^3$. For general $z_1,z_2$, this section is smooth irreducible. Denote it $D$. Its strict transform $\overline D$ has class $\pi^*(f_1+f_2)-E$ and has intersection $-1$ with $\pi^*(f_1+2f_2)-E$, so the latter is not nef. 
	The curve $\overline D$ is the base locus of the linear system determined by $\pi^*\cal O(1,2)(-E)$.

	\texttt{Step 3. Conclusion of symmetric case.} Assume $g\geq 3$. By Lemma \ref{lem:symmetricinterpolation}.(\ref{lem:symmetricinterpolationiii}), there exists a smooth curve $C_g$ through the $g$ general pairs. In particular it has multiplicity 1 at each point. Its strict transform $\overline C\subset X$ is a curve of class $\pi^*(f_1+gf_2)-E$, which has self intersection zero. Since it is also irreducible, it is nef.
		
	\texttt{Step 4. Conclusion of general case.}
	Assume that the $2g$ points $Z=\{z_1,z_2,\ldots,z_{2g}\}$ are general (including the case $g=2$). Consider the non-symmetric Cremona transform described in the last paragraph of part 3 in the proof of Lemma \ref{lem:symmetricinterpolation} below. Applying it at the points $z_1,z_2$, then at the images of $z_3,z_4$, and so on, reduces $\pi^*(f_1+gf_2)-E$ to $\pi^*f_1$. By generality, for all $1\leq i\leq g$, the images of $z_{2i-1},z_{2i}$ through any composition of the Cremona transforms above are never in the same vertical or horizontal fiber on $\bb P^1\times\bb P^1$. 
\end{proof}

\begin{lem}[Symmetric interpolation]\label{lem:symmetricinterpolation}
	Let $Z=\{z_1,z_1',\ldots,z_m,z_m'\}$ be a set of $m$ general symmetric pairs in $\bb P^1\times\bb P^1$. 
	\begin{enumerate}[\normalfont(i)]
\item\label{lem:symmetricinterpolationi} If $(n,m)\neq(1,2)$, then the
	linear system of sections of $\cal O(1,n)$ through 
	$Z$ has the expected dimension.
\item\label{lem:symmetricinterpolationii} If $(n,m)\neq(1,2)$ and $r$ is a nonnegative integer, then the linear system of sections of $\cal O(1,n)$ through $Z$ and $r$ further general points has the expected dimension.
\item\label{lem:symmetricinterpolationiii} If $(n,m)\neq(2,2)$, and the linear system in {\normalfont(\ref{lem:symmetricinterpolationi})} is nonempty, then the general divisor in this system is irreducible and smooth.
\end{enumerate}
\end{lem}

\begin{proof}
	
 Denote by $\frak b_m(n)$ the linear system in question. When no confusion is likely, we omit $n$ and denote $\frak b_m(n)=\frak b_m$.
 
 \texttt{1. Continuous variation of $\frak b_m(n)$.} For fixed $n$ and $m$, we show that the linear systems $\frak b_m$ vary continuously for general $Z$.
  The parameter space $\cal Z_m$ of ordered $m$-tuples of ordered symmetric pairs of points in $\bb P^1\times\bb P^1$ is isomorphic to $(\bb P^1\times\bb P^1)^m$. 
 Let $\cal U_m\subset\cal Z_m\times(\bb P^1\times\bb P^1)$ be the universal family. 
 The general linear systems $\frak b_m$ are captured by the general fibers of the sheaf $pr_{1*}(pr_2^*\cal O(1,n)\otimes\cal I_{\cal U_m})$ (or of its projectivization). Here $\cal I_{\cal U_m}$ is the ideal sheaf of $\cal U_m$.
 In particular, the dimension of the general $\frak b_m$ is constant, depending only on $m$. By considering the ranks of subsheaves $pr_{1*}(pr_2^*\cal O(1,a)\otimes\cal I_{\cal U_m})$ and $pr_{1*}(pr_2^*\cal O(0,a)\otimes\cal I_{\cal U_m})$ for $a\leq n$, one shows that the divisorial components of the base loci of $\frak b_m$ also vary continuously for general $Z$.
 
 \texttt{2. The cases $n=0$ and $n=1$.} 
 If $n=0$, then the projective dimension of $\frak b_m$ is 1 for $m=0$, and 0 for all $m>0$. One vertical fiber cannot contain a general symmetric pair.
  If $n=1$, then $\frak b_0$ consists of plane sections of the Segre embedding $\bb P^1\times\bb P^1\subset\bb P^3$. It has projective dimension 3.
 As in Step 2 of Proposition \ref{prop:p1xp1optimal}, for $Z=\{z_1,z_1'\}$, the system $\frak b_1$ is the pullback by $\sigma$ of the pencil of lines in $\bb P^2$ through $\sigma(z_1)=\sigma(z_1')$. In particular $\dim\frak b_1=1$.
 When $m=2$, then $\frak b_2$ is (unexpectedly) one point, corresponding to the pullback of the line through $\sigma(z_1)$ and $\sigma(z_2)$. 
 For $m>2$, the system $\frak b_m$ is empty as expected.
 
 For general choices of $Z$, the divisors constructed above are irreducible.
 
 \texttt{3. A Cremona transform on $\bb P^1\times\bb P^1$.}
 Let $z\in\bb P^1\times\bb P^1$ be a point, not on the diagonal, and let $z'\neq z$ be its reflection. Let $\rho\colon\bb P^2\dashrightarrow\bb P^1$ be the projection from $\sigma(z)\in\bb P^2$, where $\sigma\colon\bb P^1\times\bb P^1\to\bb P^2$ is the quotient map, identified with the sum map to ${\rm Sym}^2\bb P^1=\bb P^2$. Let $pr_2\colon\bb P^1\times\bb P^1\to\bb P^1$ be the second projection. Consider the rational map
 \begin{align*}
 	Cr&\colon\bb P^1\times\bb P^1\dashrightarrow\bb P^1\times\bb P^1\\
 Cr&=(\rho\circ\sigma,pr_2)
 \end{align*}
 We study some of its properties:
 \begin{enumerate}[(1)]
 	\item $Cr$ is undefined at $z$ and $z'$. Indeed $\rho$ is only undefined at $\sigma(z)=\sigma(z')$.
 	\item If $C$ is a section of $\cal O(1,1)$ through $z,z'$, then $\rho\circ\sigma$ is constant on $C\setminus\{z,z'\}$. For this, note that $C=\sigma^{-1}L$, where $L$ is a line through $\sigma(z)$.
 	\item In particular $Cr$ contracts the 2 fibers of $pr_2$ that pass through $z$ and $z'$ respectively. Clearly $pr_2$ contracts them. 
 	Let $F_{2,z}$ be the corresponding fiber of $pr_2$, and let $F_{1,z'}$ be the fiber of $pr_1$ through $z'$. Then $F_{1,z'}+F_{2,z}$ is a section of $\cal O(1,1)$ through $z,z'$, hence $\rho\circ\sigma$ is constant on it. In particular it is constant on $F_{2,z}$. (Note that $Cr$ does not also contract $F_{1,z'}$ since $pr_2$ does not contract it.)
 	\item If $x\in\bb P^1\times\bb P^1$ is any point different from $z$ and $z'$, then $Cr(x)$ and $Cr(x')$ are in the same fiber of $pr_1$. This is because $z,z',x,x'$ are contained in some section of $\cal O(1,1)$.
 	\item $Cr$ is birational. For general $x\in\bb P^1\times\bb P^1$, $\rho(\sigma(x))$ determines the section of $\cal O(1,1)$ that passes through $z$,$z'$, and through $x$, while $pr_2(x)$ determines the fiber $F_{2,x}$. Clearly the section and the fiber meet in one point unless the fiber is a component of the section, which is not the general situation.
 \end{enumerate}
 
 Finally we resolve $Cr$. Let $\pi\colon X\to\bb P^1\times\bb P^1$ be the blow-up of $z$ and $z'$ with exceptional divisors $E$ and $E'$. Contracting the strict transforms $F$ of $F_{2,z}$ and $F'$ of $F_{2,z'}$, gives a morphism $\gamma\colon X\to\bb P^1\times\bb P^1$ which is also the blow-up of two points. We have $Cr=\gamma\circ\pi^{-1}$.
 
 Because of the two blow-up structures of $X$, the N\' eron--Severi space $N^1(X)$ has two sets of bases $(\pi^*f_1,\pi^*f_2,E,E')$ and $(\gamma^*f_1,\gamma^*f_2,F,F')$. The following relations are easy consequences of the properties of $Cr$.
 \begin{align*}
 	\pi^*f_2&=E+F=E'+F'\\
 	\gamma^*f_2&=E+F=E'+F'\\
 	\gamma^*f_1&=\pi^*(f_1+f_2)-E-E' 
 \end{align*}	
In particular, the change of coordinates matrix is
\[
  \begin{pmatrix}1&0&0&0\\ 1&1&1&1\\ -1&0&-1&0\\ -1&0&0&-1\end{pmatrix}.
\]
The matrix is self-inverse, though $Cr$ is not immediately self-inverse because the source $\bb P^1\times\bb P^1$ and the target $\bb P^1\times\bb P^1$ are not canonically identified. 

The construction of $Cr$ also works in a less-symmetric situation. If $z_1,z_2$ are points not on the same horizontal or vertical fiber, then blowing-up the points and contracting the strict transforms of vertical (or of horizontal) fibers through the points gives birational $Cr\colon\bb P^1\times\bb P^1\dashrightarrow\bb P^1\times\bb P^1$. In fact one can find an automorphism that fixes $z_1$ and sends $z_2$ to $z_1'$.
When the two points are say in the same vertical fiber, and we contract the strict transforms of horizontal fibers, then the target of $Cr$ is naturally the Hirzebruch surface $\bb F_2=\bb P_{\bb P^1}(\cal O\oplus\cal O(-2))$, not $\bb P^1\times\bb P^1$. 
  
 \texttt{4. The cases $m\in\{0,1,2\}$.}
 Assume $n>1$. 
 The complete linear system $\frak b_0$ has the expected dimension $2n+1$ and irreducible general term.
 For $m\in\{1,2\}$, we perform a Cremona transform $Cr$ on $\bb P^1\times\bb P^1$ centered at $z_1,z_1'$.
  The linear system $\frak b_1$ corresponds to sections of $\pi^*\cal O(1,n)(-E-E')=\gamma^*\cal O(1,n-1)$ on $X$, so to sections of $\cal O(1,n-1)$. This gives the expected dimension of $\frak b_1$, and the irreduciblity of a general divisor.
 The linear system $\frak b_2$ similarly corresponds to sections of $\cal O(1,n-1)$ through $Cr(z_2)$ and $Cr(z_2')$. These two points live on the same vertical fiber. When $n=2$, the sections of $\cal O(1,n-1)=\cal O(1,1)$ through $Cr(z_2),Cr(z_2')$ all contain the vertical fiber through the two points, and an arbitrary horizontal fiber, giving $\dim\frak b_2=1$ as expected. We also see how irreducibility failed in this case. When $n>2$, a general section of $\cal O(1,n-1)$ intersects the vertical fiber containing the two points in $n-1\geq 2$ distinct points. Using the ${\rm PGL}(2)$ action on the second component, we can arrange that one of these points is $Cr(z_2)$, but none of the others is $Cr(z_2')$, and vice versa. Thus $\dim\frak b_2=2n-3$ as expected. By a similar construction, we can see that there exist irreducible sections of $\cal O(1,n-1)$ through $Cr(z_2)$ and $Cr(z_2')$, and avoiding the indeterminacy points of $Cr^{-1}$, hence this is the general situation.
 
 \texttt{5. Conclusion of part {\normalfont(\ref{lem:symmetricinterpolationi})}.}
 Assume $n>1$ and $m\geq 3$.
  For fixed $3\leq m\leq n+1$, assume $\frak b_{m-1}$ has the expected dimension $2(n-m+1)+1\geq 1$, but $\dim\frak b_m=2(n-m+1)$ for every general choice of $Z$. This is indeed the only choice other than the expected dimension, easily verified by picking $z_m$ outside the base locus of $\frak b_{m-1}$.
  For any such $z_m$ sufficiently general (so that $\dim\frak b_m=2(n-m+1)$ for example), consider $T\in\frak b_{m-1}$ passing through $z_m$. Let $C$ be an irreducible curve in the support of $T$ that passes through $z_m$. 
  
 By our assumption that $\dim\frak b_m=2(n-m+1)$, for all general $w$ in $C$, by replacing $z_m$ with $w$, it holds that $T$ also passes through $w'$. Then $T$ contains $C$, but also its reflection $C'$.
 If $C$ is symmetric, then it has class $cf_1+cf_2$ for some $c\geq 1$. Since $T$ has class $f_1+nf_2$, then necessarily $c=1$. 
 If $C$ is not symmetric, then similarly $C+C'$ has class $f_1+f_2$, so $C$ is a vertical or horizontal fiber. In both cases, denote $\widetilde C=C\cup C'$ (set theoretic union). It is a reduced effective symmetric cycle of class $f_1+f_2$, a section of $\cal O(1,1)$ through $z_m$ and $z_m'$. We can write $T=\widetilde C+F_{n-1}$, where $F_{n-1}$ is a sum of $n-1$ horizontal fibers (each of class $f_2$). This is an equality of cycles, and so $\widetilde C$ is uniquely determined by $T$.
 Such a decomposition exists for all sufficiently general choices of the $m$ pairs. Permuting the pairs will also produce a general ordered $m$-tuple of pairs, without changing $\frak b_m$. In particular, since $\widetilde C$ passes through $z_m$ and $z_m'$, it passes through all the $m$ pairs. This is impossible for $m\geq 3$ general pairs by the case $n=1$. 
 
 Since $\frak b_{n+1}$ is empty, so are $\frak b_m$ for all $m>n$.
 
 \texttt{6. Conclusion of part {\normalfont(\ref{lem:symmetricinterpolationii})}.}
  By (\ref{lem:symmetricinterpolationi}), the system $\frak b_m$ has the expected dimension. For any nonempty linear system, passing through one general point (e.g., not in the base locus) is a codimension 1 condition. By iterating, we obtain the claim.
 
 \texttt{7. Conclusion of part {\normalfont(\ref{lem:symmetricinterpolationiii})}.}
  In all cases where irreducibility holds, smoothness is automatic.
 Indeed any irreducible curve $C$ of class $f_1+nf_2$ satisfies $C\cdot f_2=1$,
 hence it is mapped isomorphically by the second projection onto $\bb P^1$.
 
 We assume $n>1$ and $m\geq 3$. By part (\ref{lem:symmetricinterpolationi}), $\frak b_m$ is nonempty precisely when $m\leq n$. In this case its dimension is $2(n-m)+1$.
 To prove the irreducibility of the general member of $\frak b_m$, it is enough to prove that $\frak b_m$ contains one irreducible curve. Since the first $m-1$ pairs in a set of $m$ sufficiently general pairs are also general, we have $\frak b_{m-1}\supset\frak b_{m}$. It is then enough to prove that $\frak b_n$ has an irreducible curve. 
 
 If every $C_n\in\frak b_n$ is reducible, it is necessarily of form $C_n=C_{n-1}+F$, where $C_{n-1}$ is a (potentially reducible) section of $\cal O(1,n-1)$, and $F$ is a fiber of the second projection.
 By part (\ref{lem:symmetricinterpolationii}), the section $C_{n-1}$ contains at most $2n-1$ of the points of $Z$, while $F$ contains at most one point in $Z$. Since $C_n$ passes through all of $Z$, the bounds must be sharp. By part (\ref{lem:symmetricinterpolationii}), for every $z\in Z$, there is exactly one section of $\cal O(1,n-1)$ that contains $Z\setminus\{z\}$. Clearly there exists exactly one $F$ through $z$. 
 There are then at most $2n$ choices for $C_n\in\frak b_n$.
 This contradicts the equality $\dim\frak b_n=1$ from (\ref{lem:symmetricinterpolationi}). 
 \end{proof}

\begin{rmka}[Blow-ups of $\bb P^2$]Let $g\geq 1$ and let $\pi\colon X\to\bb P^2$ be the blow-up of $2g$ general points $z_1,\ldots,z_{2g}$ with exceptional divisors $E_1,\ldots,E_{2g}$ respectively. Let $H$ be the class of a line in $\bb P^2$. Then
	\[\pi^*gH-(g-1)E_1-E_2-\cdots-E_{2g}\in\Nef(X).\]
Indeed this class can be reduced by a sequence of Cremona transforms
to $\pi^*H-E_1$. This result is equivalent to the general case of Proposition \ref{prop:p1xp1optimal} via the isomorphism between $\bb P^2$ blown-up at $2g+1$ points $z_0,\ldots,z_{2g}$ and $\bb P^1\times\bb P^1$ blown-up at $2g$ points. The exceptional divisor $E_0$ over $z_0$ is considered with coefficient 0 in the class above, so it can be blown-down.

\end{rmka}

We now construct nef classes on $C\times C$ by degenerating from simple covers of $\bb P^1$. This follows an idea of \cite{kouvidakis}.
Recall that a finite map $f\colon C\to\bb P^1$ is called a \emph{simple branched cover}
if any fiber of $f$ has at most one ramification point $c$, and if $f$ is given locally
around any such $c$ by the map $x\mapsto x^2$.
For example hyperelliptic pencils are simple. 

\begin{ex}[Simple branched covers]Let $C$ be a curve of genus $g\geq 1$.
	Assume that $C$ admits a simple branched cover $f\colon C\to\bb P^1$ 
	of degree $2\leq d\leq\lfloor\sqrt g\rfloor+1$. 
	\[\text{If }a,b\geq d\text{, then }af_1+bf_2-\delta\text{ is nef if and only if }a+b\geq\frac{2g}{d-1}+2\]
	(Following \cite{kouvidakis} and \cite[Theorem 1.5.8]{laz04}, consider $T$ the closure of the complement
	of the diagonal in $C\times_{\bb P^1}C$. It is irreducible of class $df_1+df_2-\delta$ in $C\times C$.
	Its self-intersection is $2\cdot((d-1)^2-g)\leq 0$. For example if $f$ is a hyperelliptic pencil, then $T$ is the graph of the induced hyperelliptic involution.
	For $A,B,C\geq 0$,
	\[(A+dC)f_1+(B+dC)f_2-C\delta=Af_1+Bf_2+C[T]\]
	is nef if and only if the intersection with $T$ is nonnegative, i.e.,
	$(d-1)(A+dC+B+dC-2C)\geq 2g C$.
	If $C>0$, after setting $a=\frac AC+d$ and $b=\frac BC+d$, we obtain the claim.)
	
	When $d>\lfloor\sqrt g\rfloor+1$, and $a,b\geq d$, the class $af_1+bf_2-\delta$ is ample.
	\qed
\end{ex}

We obtain the following extension of the result of Kouvidakis \cite[Theorem 2]{kouvidakis}.

\begin{cor}\label{cor:genKouvidakis}Let $C$ be a \emph{very general} curve of genus $g\geq 1$.
	If $2\leq d\leq \lfloor\sqrt g\rfloor+1$ is an integer and $a,b\geq d$ 
	satisfy $a+b\geq2+\frac{2g}{d-1}$, then $af_1+bf_2-\delta$ is nef.
	In particular
	\begin{equation}\label{eq:genKouvidakis}
		d\:\!f_1 + \biggl( 2 + \frac{2g}{d-1}-d \biggr) f_2 - \delta \in
		\operatorname{Nef}(C \times C)
	\end{equation}
\end{cor}


\begin{proof}By the Riemann existence theorem, 
	for any degree $d\geq 2$ there exists a curve 
	$C_d$ of genus $g$ admitting a simple branched cover $C_d\to\bb P^1$ of degree $d$. 
	From the previous example we deduce that $af_1+bf_2-\delta$ is nef. 
	Since nefness is a very general condition in families, the result extends to very general curves.
	For the last statement set $a=d$ and note that $2 + \frac{2g}{d-1}-d\geq d$. 
\end{proof}

\begin{rmka}
	When $C$ is \emph{very general}, the nefness of the classes in Theorem \ref{thrm:prodcurvescgeneral}.(\ref{thrm:prodcurvescgenerali}) can be deduced from Corollary \ref{cor:genKouvidakis}. 
	Some of the classes in Theorem \ref{thrm:prodcurvescgeneral}.(\ref{thrm:prodcurvescgeneralii}) are better, e.g., when $d=2g-2$. However, for large $g$ (the Mathematica software suggests $g\geq 15$), they are all in the convex span of the classes in \eqref{eq:Vojta} and those in Corollary \ref{cor:genKouvidakis}.
	
	It is conceivable that for some countable
	union of families of curves inside $\scr M_g$ Corollary \ref{cor:genKouvidakis} fails,
	while Theorem \ref{thrm:prodcurvescgeneral} does not. 
\end{rmka}

\subsection{General technical results used in our proofs}

\begin{prop}\label{prop:curveapprox}
  Let $\rho \colon X \to C$ be a flat surjective morphism between projective
  varieties, where $C$ is a nonsingular projective curve over an algebraically closed field.
  Let $\mathcal{L}$ be a line bundle on $X$, and let $f$ be the class of a
  fiber of $\rho$.
  
  \begin{enumerate}[\normalfont(i)]
  	
  	\item\label{prop:curveapproxeasy}
  	If $\cal L$ is nef on every fiber of $\rho$ and globally generated on a general fiber of $\rho$,
  	then
  	\[c_1(\cal L)-\bigl(\overline{\mu}_{\min}(\rho_*\cal L)\bigr)\cdot f\text{ is nef on }X.\]
  	
  	\item\label{prop:curveapproxharder} If $\cal L$ is $\rho$-ample, then
  	\begin{equation}\label{eq:curveapprox}
  		\sup\bigl\{t \bigm\vert c_1(\mathcal{L}) - tf\ \text{is nef}\bigr\}
  		= \lim_{n \to \infty}
  		\frac{\overline{\mu}_{\mathrm{min}}(\rho_*\mathcal{L}^{\otimes n})}{n}.
  	\end{equation}

\end{enumerate}
\end{prop}
\begin{proof}	
	(\ref{prop:curveapproxeasy}). The assumption implies by cohomology and base change that
	the natural map $\rho^*\rho_*\cal L\to \cal L$ is surjective on the general fiber of $\rho$. We thus have an exact complex
	\[\rho^*\rho_*\cal L\longrightarrow\cal L\longrightarrow Q\longrightarrow 0,\]
	where $Q$ is supported in at most finitely many fibers of $\rho$.
	Since $\cal L$ is nef on the fibers and $Q$ is a direct sum of quotients of $\cal L$ restricted to fibers, we deduce that $Q$ is a nef coherent sheaf. 
	It is invariant under twisting by classes of form $\rho^*D$ with $D$ an $\bb R$-divisor on $C$, since these are trivial on fibers. 
	The twisted sheaf $\rho^*\rho_*\cal L\langle-\rho^*\overline{\mu}_{\min}(\rho_*\cal L)f\rangle$
	is nef by Lemma \ref{lem:bp14}. The same is true of its (twisted) image in $\cal L\langle-\rho^*\overline{\mu}_{\min}(\rho_*\cal L)f\rangle$.
	We deduce that the latter is an extension of nef twisted coherent sheaves.
	\cite[Remark 3.4]{fm19} and \cite[Lemma 3.31]{fm19} prove that such extensions are nef. 
	
	One can also argue by blowing-up the ideal sheaf $\cal I$ on $X$
	such that $\cal I\otimes\cal L$ is the image of the natural map $\rho^*\rho_*\cal L\to\cal L$. 
	
	\medskip
	
	(\ref{prop:curveapproxharder}). We first show that the right-hand side of \eqref{eq:curveapprox} is indeed a
  limit.
  Let $n_0$ be an integer such that $\rho_*\mathcal{L}^{\otimes n}$ is a vector
  bundle for every $n \ge n_0$, and such that the natural maps
  \begin{equation}\label{eq:relativemultmaps}
    \rho_*\mathcal{L}^{\otimes n} \otimes \rho_*\mathcal{L}^{\otimes m}
    \longrightarrow \rho_*\mathcal{L}^{\otimes (n+m)}
  \end{equation}
  are surjective for all $n \ge n_0$, $m \ge n_0$.
  Note that such an $n_0$ exists by cohomology and base change and by
  \cite[Example 1.8.4.(ii)]{laz04}, respectively.
  We then have
  \[
    \overline{\mu}_{\mathrm{min}}(\rho_*\mathcal{L}^{\otimes n}) +
    \overline{\mu}_{\mathrm{min}}(\rho_*\mathcal{L}^{\otimes m})
    = \overline{\mu}_{\mathrm{min}}(\rho_*\mathcal{L}^{\otimes n} \otimes
    \rho_*\mathcal{L}^{\otimes m})
    \le \overline{\mu}_{\mathrm{min}}(\rho_*\mathcal{L}^{\otimes (n+m)})
  \]
  for all $n \ge n_0$, $m \ge n_0$.
  The equality holds by \cite[Corollary 3.7 and p.\ 464]{Miy87}. 
  For the inequality, in characteristic zero use \eqref{eq:relativemultmaps} and \eqref{eq:slope+}.
  In positive characteristic, the same argument works to show the inequality
  above after taking a large enough Frobenius pullback of the quotient map
by \cite[Theorem 2.7]{Langer04}.
  Finally, the sequence
  $\{-\overline{\mu}_{\mathrm{min}}(\rho_*\mathcal{L}^{\otimes
  n})\}_{n=1}^\infty$ is a sequence satisfying the hypothesis of de Bruijn and
  Erd\H{o}s's version of Fekete's lemma \cite[Theorem 23]{deBruijnErdos} for the constant function
  \[
    \varphi(t) = \max\biggl\{0,\max_{1 \le n,m \le n_0}
    \Bigl\{\overline{\mu}_{\mathrm{min}}(\rho_*\mathcal{L}^{\otimes n}) +
    \overline{\mu}_{\mathrm{min}}(\rho_*\mathcal{L}^{\otimes m}) -
    \overline{\mu}_{\mathrm{min}}(\rho_*\mathcal{L}^{\otimes (n+m)})\Bigr\}
    \biggr\}
  \]
  in their notation,
  hence the right-hand side of \eqref{eq:curveapprox} is indeed a limit.
  \par We now show that 
  \[
    \sup\bigl\{t \bigm\vert c_1(\mathcal{L}) - tf\ \text{is nef}\bigr\} \ge
    \frac{\overline{\mu}_{\mathrm{min}}(\rho_*\mathcal{L}^{\otimes n})}{n}
  \]
  for all $n$ sufficiently large.
  By \cite[Theorem 1.7.6.(iii)]{laz04} and the $\rho$-ampleness of
  $\mathcal{L}$, we may assume that $n$ is sufficiently large such that
the natural map
  $\rho^*\rho_*\mathcal{L}^{\otimes n} \to \mathcal{L}^{\otimes
  n}$ is surjective.
  We may also assume that $\rho_*\mathcal{L}^{\otimes n}$ is locally free. 
  Choose a closed point $c \in C$.
  Letting $a\coloneqq -\overline{\mu}_{\mathrm{min}}(\rho_*\mathcal{L}^{\otimes n})$,
  we see that
  \begin{align*}
    \sup\biggl\{t \biggm\vert c_1(\mathcal{L}) + \frac{a}{n}f - tf\ \text{is
    nef}\biggr\} &= \frac{a}{n} + \sup\bigl\{t \bigm\vert c_1(\mathcal{L}) - tf\
    \text{is nef}\bigr\}\\
    \frac{\overline{\mu}_{\mathrm{min}}(\rho_*\mathcal{L}^{\otimes
    n}\langle ac\rangle)}{n} &= \frac{a}{n} +
    \frac{\overline{\mu}_{\mathrm{min}}(\rho_*\mathcal{L}^{\otimes
    n})}{n} = 0
  \end{align*}
  hence it suffices to show that
  \begin{equation}\label{eq:curveapproxineq}
    \sup\biggl\{t \biggm\vert c_1(\mathcal{L}) + \frac{a}{n}f - tf\ \text{is
    nef}\biggr\} \ge 0.
  \end{equation}
  Since
  $\overline{\mu}_{\mathrm{min}}(\rho_*\mathcal{L}^{\otimes n}\langle ac\rangle) = 0$,
  we see that $\rho_*\mathcal{L}^{\otimes n}\langle ac\rangle$ is a nef twisted bundle 
by Lemma \ref{lem:bp14}.
  Using the surjection $\rho^*\rho_*\mathcal{L}^{\otimes n}
  \twoheadrightarrow \mathcal{L}^{\otimes n}$, we have the
  commutative diagram
  \[
    \xymatrix{
      X \ar@{^{(}->}[r]\ar[dr]_{\rho} &
      \mathbb{P}\bigl(\rho_*(\mathcal{L}^{\otimes n})\langle ac\rangle\bigr)
      \ar[d]^{\pi}\\
      & C
    }
  \]
  where $\mathcal{O}_{\mathbb{P}(\rho_*(\mathcal{L}^{\otimes n})\langle ac\rangle)}(1)$ is nef, hence so is $\mathcal{O}(1)\rvert_X =
  \mathcal{L}^{\otimes n}\langle af\rangle$, and
  \eqref{eq:curveapproxineq} follows.
  \par It remains to show that the inequality $\le$ holds in
  \eqref{eq:curveapprox}.
  Choose a closed point $c \in C$, and let $a$ be sufficiently large such
  that $\mathcal{L}\langle af\rangle$ is ample.
  We then see that replacing $\mathcal{L}$ by $\mathcal{L}\langle ac\rangle$
  results in both sides of \eqref{eq:curveapprox} increasing by $a$, and it
  therefore suffices to consider the case when $\mathcal{L}$ is ample.
  Now let $t_0 > 0$ be the value of the supremum on the left-hand side of
  \eqref{eq:curveapprox}, and fix a real number $\varepsilon > 0$.
  Choose integers $u,v\ge 1$ such
  that $u/v + \varepsilon > t_0 > u/v$.
  We then see that
  \[
    \frac{\overline{\mu}_{\mathrm{min}}(\rho_*\mathcal{L}^{\otimes
    vn}(-unc))}{n} = -u + v \cdot
    \frac{\overline{\mu}_{\mathrm{min}}(\rho_*\mathcal{L}^{\otimes n})}{n}.
  \]
  Since $\mathcal{L}^{\otimes v}(-uf)$ is ample, Lemma \ref{lem:CM} below
  implies
  \[
    \lim_{n \to
    \infty}\frac{\overline{\mu}_{\mathrm{min}}(\rho_*\mathcal{L}^{\otimes
    vn}(-unc))}{n} \ge 0
  \]
  by the fact that $\rho_*\mathcal{L}^{\otimes vn}(-unc)$ is nef for $n$
  sufficiently large, and then by Lemma \ref{lem:bp14}.
  We therefore have
  \[
    \lim_{n \to \infty}
    \frac{\overline{\mu}_{\mathrm{min}}(\rho_*\mathcal{L}^{\otimes
    n})}{n} \ge \frac{u}{v} > t_0 - \varepsilon.
  \]
  Since $\varepsilon$ was arbitrary, the inequality $\le$ holds in
  \eqref{eq:curveapprox}.
\end{proof}

\begin{lem}\label{lem:CM}Let $\rho\colon Y\to X$ be a morphism of projective schemes, and let $\cal L$ be an ample invertible sheaf on $Y$.
Let $\cal F$ be a coherent sheaf on $X$.
Then $\cal F\otimes\rho_*\cal L^{\otimes n}$ is ample and globally generated for all $n$ sufficiently large.
\end{lem}
\begin{proof}
Let $A$ be a very ample divisor on $X$ such that there exists a surjection
$\bigoplus\cal O_X(-A)\twoheadrightarrow\cal F$.
Since ampleness and global generation descend to quotients, it is enough
to prove the lemma for $\cal F=\cal O_X(-A)$.
With the usual arguments of Castelnuovo--Mumford regularity \cite[Theorem 1.8.5]{laz04}, it is enough to prove
that if $A$ is a very ample divisor on $X$, then 
$\rho_*\cal L^{\otimes n}$ is $-2$-regular with respect to $A$,
i.e., $H^i\bigl(X,\rho_*\cal L^{\otimes n}(-(2+i)A)\bigr)=0$ for all $i>0$ for all $n$ sufficiently large. 
This is because in this case $\rho_*\cal L^{\otimes n}(-2A)$ is globally generated, hence $\rho_*\cal L^{\otimes n}(-A)$ is ample and globally generated.

Since $\cal L$ is ample, it is in particular also $\rho$-ample. 
Hence for $n$ large, we have $R^i\rho_*\cal L^{\otimes n}=0$ for all $i>0$. 
The Leray spectral sequence and the projection formula show that 
$H^i\bigl(X,\rho_*\cal L^{\otimes n}(-(2+i)A)\bigr)=H^i\bigl(Y,\cal L^{\otimes n}\otimes\rho^*(-(2+i)A)\bigr)$.
The ampleness of $\cal L$ and Serre vanishing show that these cohomology groups are $0$.
\end{proof}

Finally, we show that Proposition \ref{prop:curveapprox}.(\ref{prop:curveapproxharder}) extends to a more general setting:

\begin{prop}\label{prop:positivitybypushforward}Let $\rho\colon Y\to X$ be a morphism of projective schemes over an algebraically closed field. 
Let $\cal L$ be a $\rho$-ample line bundle on $Y$.
For $\cal F$ a coherent sheaf on $X$, and $H$ an ample line bundle on $X$, denote 
$\nu_H(\cal F)\coloneqq  \sup\{t \mid \cal F\langle -tH\rangle\mbox{ is nef}\}.$
Then
\[\sup\bigl\{t\st c_1(\cal L)-t\rho^*H\mbox{ is nef}\bigr\}=\lim_{n\to\infty}\frac{\nu_H(\rho_*\cal L^{\otimes n})}n.\]
\end{prop}
\begin{proof}
The sequence $\nu_H(\rho_*\cal L^{\otimes n})>-\infty$ is superadditive (in the
weaker sense in the proof of Proposition \ref{prop:curveapprox}.(\ref{prop:curveapproxharder})) by
$\rho$-ampleness, hence the limit exists by de Bruijn and Erd\H{o}s's version of Fekete's lemma \cite[Theorem 23]{deBruijnErdos}.
Since $\cal L$ is $\rho$-ample, for sufficiently large $n$, we have inclusions $Y\hookrightarrow\bb P(\rho_*\cal L^{\otimes n})$
such that $\cal O_{\bb P_X(\rho_*\cal L^{\otimes n})}(1)|_Y=\cal L^{\otimes n}$.
It follows that the inequality ``$\geq$'' holds.

For the reverse inequality, note as in Proposition \ref{prop:curveapprox}.(\ref{prop:curveapproxharder}) that both sides translate by $t_0$ when replacing 
$\cal L$ by $\cal L\langle t_0\rho^*H\rangle$ for $t_0\in\bb Q$ (with the
understanding that we only consider sufficiently divisible $n$ in the right-hand
side).
Without loss of generality, we may assume that $\cal L$ is ample on $Y$. 
As in Proposition \ref{prop:curveapprox}, we reduce to proving that $\rho_*\cal L^{\otimes n}$ is globally generated for large $n$,
which follows from Lemma \ref{lem:CM}.
\end{proof}

\section{Products of arbitrarily many factors}\label{sec:higherproducts}

In this section we work over $\bb C$.
Let $n\geq 2$ be an integer. It is also interesting to study $\Nef(C^n)$.
To our knowledge, no conjecture on the shape of this cone has been made in the literature for $n\geq 3$.
In fact it is quite a large cone.

\begin{lem}If $g\geq 1$ and $C$ is very general in moduli, then 
$N^1(C^n)$ has dimension $\binom{n+1}{2}$. A basis is given by the class $f_i$ of the fibers of the projection from $C^n$ onto the $i$-th factor for every $i$,
and by the classes $\delta_{ij}$ of the big diagonals $\Delta_{ij}=\{(x_1,\ldots,x_n)\in C^n\st x_i=x_j\}$.
\end{lem}

Pulling back by the projections $pr_{ij}\colon C^n\to C^2$ gives nef classes on $C^n$.

\begin{ex}\label{ex:trivialexampleCn}With assumptions as in Theorem \ref{thrm:prodcurvescgeneral}.(\ref{thrm:prodcurvescgenerali}), on general curves we have
\[\frac{(n-1)d}{d-g}f_1+d\:\!f_2+\cdots+d\:\!f_n-\delta_{12}-\cdots-\delta_{1n}=\sum_{i=2}^n\biggl(\Bigl(1+\frac{g}{d-g}\Bigr)f_1+df_i-\delta_{1i}\biggr)\in\Nef(C^n).\]
Indeed, each summand on the right is the pullback via $pr_{1i}$ of a nef class on $C\times C$.
On an arbitrary curve, using \eqref{eq:Vojta}, we obtain
$\sum_{i=2}^n\bigl(\bigl(1+\frac{g}{a-1}+(g-1)(a-1)\bigr)f_1+af_i-\delta_{1i}\bigr)\in\Nef(C^n)$ for all $a>1$.\qed
\end{ex}

Our main result of the section is the following

\begin{thrm}\label{thrm:mainCn}Let $C$ be a smooth complex projective curve of genus $g>0$ and let $n\geq 2$.
Then \[\frac{(n-1)d}{d-g}f_1+d\cdot\sum_{i=2}^nf_i-\sum_{1\leq i<j\leq n}\delta_{ij}\in\Nef(C^n)\]
if one of the following holds:
\begin{enumerate}[\normalfont(i)]
\item\label{thrm:mainCni} $d\geq 2g+n$, or $d\geq \max\{2n+g,2g\}$, or
\item\label{thrm:mainCnii} $n\geq 2g$ and $d\geq g+n-1$.
\end{enumerate}
\end{thrm}

\noindent The proof will make use of the rich geometry of the symmetric products of $C$.
We recall some notation and classical results about these. 
Denote by
\[C_n\coloneqq\operatorname{Hilb}^nC=C^n/\mathfrak S_n\] 
the $n$-th symmetric product of $C$ with quotient map $\pi\colon C^n\to C_n$. 
It is also the space of effective divisors $D$ on $C$ with $\deg D=n$.
Let $\Delta=\Delta_n$ be the big diagonal on $C_n$, the image through $\pi$ of $\Delta_{ij}$ for any $i<j$. 
It is the ramification locus of $\pi$,
hence there exists a (non-effective) Cartier divisor on $C_n$ denoted $\frac{\Delta_n}2$ such that $\pi^*\frac{\Delta_n}2$ is the branching divisor $\sum_{i<j}\Delta_{ij}$ and $2\cdot\frac{\Delta_n}2=\Delta_n$.

Let $x$ and $\delta$ be the classes of $c_0+C_{n-1}$ and of $\Delta_n$ respectively in $N^1(C_n)$. 
If $C$ is very general, then $x$ and $\delta$ are a basis of $N^1(C_n)$. 

The cone $\Nef(C_n)$ has been previously studied by \cite{Pacienza}. 
Our methods do not offer improvements here, since we usually exploit
the existence of a nonconstant map to a curve. 
This is in line with our results above in the case $n=2$.
Instead we focus on $\Nef(C\times C_{n-1})$. 
Any of the nef classes that we construct lift to $\frak S_{n-1}$-symmetric (but not necessarily $\frak S_{n}$-symmetric) nef classes on $C^{n}$.
Consider the diagram
\[
\xymatrix{
Z_{n-1}\ar@{^{(}->}[r]& C\times C_{n-1}\ar[r]^-q\ar[d]_p&C_{n-1}\\ 
 & C & 
}\]
where $p$ and $q$ are the two projections, and $Z=Z_{n-1}$ is the universal family $\{(c,D)\st c\in{\rm Supp}\, D\}$.
Denote by $z$ its class in $N^1(C\times C_{n-1})$. 

\begin{lem}If $n\geq 3$, if $g\geq 1$ and $C$ is very general in moduli, then $N^1(C\times C_{n-1})$ is 4-dimensional, generated by
the fiber $f$ of the first projection $p$, by $q^*x$ and $q^*\delta$, and by $z$. 
\end{lem}
\begin{proof}
Modulo pullbacks from either factor, a divisor on $C\times C_{n-1}$ can be identified with a morphism $C_{n-1}\to J(C)$.
From the universality property of the Albanese variety, this is equivalent to an
element of
$\operatorname{Hom}(\operatorname{Alb}(C_{n-1}),J(C))=\operatorname{End}(J(C))$.
The latter is $\bb Z$ for $C$ very general.
We used $\operatorname{Alb}(C_{n-1})=J(C)$. 
\end{proof}

For $L$ a divisor on $C$, consider the tautological divisors on $C_{n-1}$:
\begin{itemize} 
\item $T_L\coloneqq\pi(pr_i^*L)$, where $pr_i\colon C^{n-1}\to C$ is any of the projections. 
We have $\pi^*T_L=L^{\boxtimes (n-1)}\coloneqq\sum_i pr_i^*L$. 
If $L=c_0$ is a point, then $T_L=\{c_0+D\st D\in C_{n-2}\}=c_0+C_{n-2}\subset C_{n-1}$. 
\item $N_L\coloneqq T_L-\frac{\Delta_{n-1}}2$. 
It is the determinant of the tautological bundle $E_L\coloneqq q_*\cal O_{Z_{n-1}}(p^*L)$ (not to be confused with the bundle $E_L=q_*(p^*\cal O_C(L)\otimes\cal O_{C\times C}(\Delta))$ from \cite{elstable} appearing in the previous section). For example, $K_{C_n}=N_{K_C}$.
\end{itemize}

Part (\ref{lem:transformSchuri}) in the next result is an important computation for the proof of Theorem \ref{thrm:mainCn}.
Part (\ref{lem:transformSchurii}) will not be used, but we find that the diagram used in its proof (taken from \cite{ELgonality}) is instructive. 
See Definition \ref{defn:higherconormal} for the definition of $M^{(m-1)}(L)$.

\begin{lem}\label{lem:transformSchur}Let $C$ be a smooth complex projective curve, and let $n\geq 2$ and $m\geq 1$.
If $L$ is a divisor on $C$, then
\[pr_{1*}\cal O\Bigl(pr_{23\ldots n}^*L^{\boxtimes(n-1)}-m\cdot
\sum_{i=2}^n\Delta_{1i}\Bigr)=\bigotimes^{n-1}M^{(m-1)}(L).\]
Consequently,
\begin{enumerate}[\normalfont(i)]
\item\label{lem:transformSchuri} $p_*\cal O(q^*T_L-mZ)=\Sym^{n-1}M^{(m-1)}(L)$ and 
$p_*\cal O(q^*N_L-mZ)=\bigwedge^{n-1}M^{(m-1)}(L)$.
\item\label{lem:transformSchurii} $p_*\cal O_Z(q^*T_L)=\Sym^{n-2}H^0(C,L)\otimes\cal O_C(L)$
and $p_*\cal O_Z(q^*N_L)=\bigwedge^{n-2}M_L\otimes\cal O_C(L)$.
\end{enumerate}
\end{lem}
\begin{proof}
We can see $C^n$ as
\[
  \underbrace{(C\times C)\times_C\ldots\times_C(C\times C)}_{n-1\text{ times}},
\]
where the fiber product is always over the first projection. Then $pr_{23\ldots n}^*L^{\boxtimes(n-1)}-m\cdot\sum_{i=2}^n\Delta_{1i}$ is the relative box product pullback of $pr_2^*L-m\Delta$ from each $C\times C$ factor. Since $pr_{1*}\cal O\bigl(pr_2^*L-m\Delta\bigr)=M^{(m-1)}(L)$,
the claim follows by the projection formula. 

For (\ref{lem:transformSchuri}) consider the diagram
\[\xymatrix{
C^n\ar[r]^-{pr_{23\ldots n}}\ar[dr]^{1_C\times\pi}\ar[dd]_{pr_1}&C^{n-1}\ar[dr]^{\pi}& \\
& C\times C_{n-1}\ar[r]_-q\ar[dl]^p&C_{n-1}\\
C& &
}\]
The permutation group $\frak S_{n-1}$
acts naturally on the last $n-1$ components of $C^n$,
and trivially on the first component $C$ and on the quotient $C_{n-1}$. 
Then the maps $1_C\times\pi$, $p$, and $pr_1$ are $\frak S_{n-1}$-equivariant. 
We have 
$pr_{23\ldots n}^*L^{\boxtimes(n-1)}-m\cdot\sum_{i=2}^n\Delta_{1i}=(1_C\times\pi)^*(q^*T_L-mZ)$.
The associated line bundle can be linearized in two natural ways
via the identity and alternating representation.
With these linearizations, it descends to $C\times C_{n-1}$
as $q^*T_L-mZ$ and $q^*N_L-mZ$ respectively. 
Part (\ref{lem:transformSchuri}) follows from \cite[Proposition 2.3]{Sheridan}
since $\Sym^{n-1}M^{(m-1)}(L)$ and $\bigwedge^{n-1}M^{(m-1)}(L)$
are the respective invariants for 
$pr_{1*}\cal O\bigl(pr_{23\ldots n}^*L^{\boxtimes(n-1)}-m\sum_{i=2}^n\Delta_{1i}\bigr)=\bigotimes^{n-1}M^{(m-1)}(L)$
under these actions.
\vskip.25cm
For part (\ref{lem:transformSchurii}), consider the commutative diagram 
\[\xymatrix@C=1pc{
 & C_{n-1} & \\
C\times C_{n-2}\ar[ur]^{\sigma}\ar@{^{(}->}[rr]^{\jmath}\ar[dr]_{p_1}\ar[d]_{p_2} & & C\times C_{n-1}\ar[lu]_q\ar[dl]^p\\
C_{n-2}& C &
}\]
where $\sigma(c,D_{n-2})=c+D_{n-2}$, 
where $p_1$ and $p_2$ are the first and second projection, 
and $\jmath(c,D_{n-2})=(c,c+D_{n-2})$.
The map $\jmath$ can be identified with the inclusion of the universal family $Z$. We compute that 
$\sigma^*(c_0+C_{n-2})=p_1^*(c_0)+p_2^*(c_0+C_{n-3})$
which extends to $\sigma^*T_L=p_1^*L+p_2^*T_L$,
and hence
\[p_*\cal O_Z(q^*T_L)=p_{1*}\cal O(\sigma^*T_L)=H^0(C_{n-2},T_L)\otimes\cal O_C(L)\]
by the projection formula and flat base change for cohomology. 

\par For $n=3$, we observe $\sigma^*(\frac{\Delta}2)=\Delta_{12}=Z_1$, hence
\[p_*\cal O_Z(q^*N_L)=p_{1*}\cal O\bigl(p_2^*L-{\Delta_{12}}\bigr)\otimes\cal O_C(L)=M_{L}\otimes\cal O_C(L).\]
\par For $n>3$ we compute $\sigma^*(\frac{\Delta}2)=Z_{n-2}+p_2^*\frac{\Delta_{n-2}}2$.
Then by part (\ref{lem:transformSchuri}), 
\[p_*\cal O_Z(q^*N_L)=p_{1*}\cal O\left(p_2^*N_L-Z_{n-2}\right)\otimes\cal O_C(L)
= \bigwedge^{n-2} M_L \otimes \cal O_C(L).\qedhere\]
\end{proof}





Starting with a divisor $L=L_d$ of degree $d$ on $C$,
the classical approach for constructing nef divisors on $C_n$
has been to exploit the properties of tautological sheaves $E_L=q_*\cal O_Z(p^*L)$ and the divisors $T_L$ and $N_L$ by working with them on $C_n$ directly (see \cite{ELgonality,Sheridan}).
The semistability of $E_{L_d}$ is known for large $d$ (see \cite{MistrettaStability}).
What we are missing is an understanding of the positivity of twists
of $E_{L}$ as is necessary for Proposition \ref{prop:positivitybypushforward}.
In our approach, where we work with vector bundles on $C$ instead,
the positivity of twists is completely determined by semistability.

\begin{proof}[Proof of Theorem \ref{thrm:mainCn}]
In all cases we prove that
\begin{equation}\label{eq:nefclaimCn}q^*\Bigl(dx-\frac{\delta}2\Bigr)-z+\frac{(n-1)d}{d-g}f\in\Nef(C\times C_{n-1}).
\end{equation}
Via $\pi^*$, it then lifts to 
$\frac{(n-1)d}{d-g}f_1+d\:\!f_2+\cdots+d\:\!f_{n+1}-\sum_{1\leq i<j\leq n}\delta_{ij}\in\Nef(C^{n}).$
\medskip 

Recall that a divisor $D$ on a smooth projective curve $C$ is called $k$-very ample if the evaluation map $H^0(C,D)\to H^0(Y_{k+1},D|_{Y_{k+1}})$ is surjective for all effective divisors $Y_{k+1}$ of degree $k+1$ on $C$.
Equivalently, the tautological bundle $E_D=q_*\cal O_{Z_{k+1}}(p^*D)$ on $C_{k+1}$ is globally generated. 
In particular, if $D$ is $k$-very ample, then $N_D$ is globally generated on $C_{k+1}$. \cite{CG90} prove that if $D$ is furthermore $k+1$-very ample, then $N_D$ is in fact very ample on the same $C_{k+1}$.

It is immediate that if $\deg D\geq 2g+p$, then $D$ is $p$-very ample. 

\medskip

(\ref{thrm:mainCni}). Let $L$ be a divisor of degree $d\geq 2g+n$. Then $L-c_0$ is $n-1$-very ample for all $c_0\in C$, hence 
$N_{L-c_0}=N_L-(c_0+C_{n-2})$ is very ample on $C_{n-1}$. 
By Lemma \ref{lem:transformSchur}, 
\begin{equation*}p_*\cal O(q^*N_L-Z)=\bigwedge^{n-1}M_L.\end{equation*}
Since $M_{L}$ is semistable for $d\geq 2g$ (see \cite[Proposition 3.2]{elstable}), 
so are its tensor powers $\bigotimes^{n-1}M_L$ by \cite[Corollary 6.4.14]{laz042}, 
hence so are any direct summands of these tensor powers, e.g., 
$\bigwedge^{n-1}M_L$.
We conclude that $\bigwedge^{n-1}M_{L}$
is semistable of slope $(n-1)\cdot\mu(M_{L})=-\frac{(n-1)d}{d-g}$.
From Proposition \ref{prop:curveapprox}.(\ref{prop:curveapproxeasy}) we find that $q^*N_L-Z+\frac{(n-1)d}{d-g}\cdot f$ is nef. 
Note that $A=N_L$ has class $dx-\frac{\delta}2$, 
which gives \eqref{eq:nefclaimCn}.
\medskip

Let $L$ be a \emph{general} divisor of degree $d\geq \max\{2n+g,2g\}$.
Then $M_L$ is semistable of slope $-\frac d{d-g}$, in particular $\bigwedge^{n-1}M_L$ is semistable of slope $-\frac{(n-1)d}{d-g}$. It remains to prove that $L-c_0$ is $n-1$-very ample for general $c_0\in C$. This follows from Lemma \ref{lem:KMsteal} below.

\medskip

(\ref{thrm:mainCnii}). 
Let $L=L_d$ be a divisor on $C$ of degree $d$.
Fix $c_0\in C$.
We want to show that $q^*N_L-Z+\frac{(n-1)d}{d-g}p^*c_0$ is nef on $C\times C_{n-1}$. 
We begin with a reminder on the Abel--Jacobi map:
\begin{equation}\label{eq:abeljacobium}
	\begin{aligned}
		u_m\colon C_{m} &\longrightarrow \operatorname{Pic}^0(C)=J(C)\\
		D &\longmapsto \cal O_C\bigl(D-mc_0\bigr)
	\end{aligned}
\end{equation}
It is a projective bundle for $m\geq 2g-1$ \cite[p.\ 309, Proposition
2.1(i)]{acgh}.
Here, $c_0\in C$ is our fixed point.


Let $\theta$ be the principal polarization on $J(C)$ obtained as the image through the Abel--Jacobi map $u_{g-1}$.
Denoting by $\tau_{\alpha}$ the translation by $\alpha\in J(C)$, 
it induces an isomorphism 
\begin{align*}
	\kappa\colon J(C)&\longrightarrow\operatorname{Pic}^0(J(C))\\
	\alpha&\longmapsto \tau_{\alpha}^*\theta-\theta
\end{align*}
whose inverse is $u_1^*\colon\operatorname{Pic}^0(J(C))\to\operatorname{Pic}^0(C)=J(C)$.
We make the following claims:
\medskip 

\texttt{Main claim.} For some large $N>0$ and sufficiently general $\alpha_i\in{\rm Pic}^0(C)$ with $i\in\{1,2,\ldots,N\}$,
	the natural map
	\begin{equation}\label{eq:continuousglobalgeneration}
		\bigoplus_{i=1}^Np^*p_*\cal O(q^*N_{L+\alpha_i}-Z)\otimes \cal O(q^*T_{\alpha_i^{\vee}})\longrightarrow\cal O(q^*N_L-Z)
	\end{equation}
	is surjective along the general fiber of $p$. The map is obtained by summing the compositions
	$p^*p_*\cal O(q^*N_{L+\alpha_i}-Z)\otimes \cal O(q^*T_{\alpha_i^{\vee}})
	\to
	\cal O(q^*N_{L+\alpha_i}-Z)\otimes\cal O(q^*T_{\alpha_i^{\vee}})=\cal O(q^*N_L-Z)$.
\medskip
	
\texttt{Claim 2.} If $\alpha\in J(C)$, then $T_{\alpha}$ is numerically trivial.
	In fact $T_{\alpha}=u^*(\kappa(\alpha))$, where $u=u_{n-1}$. 
(Since $u$ is a projective bundle map, $u^*$ is an isomorphism at the level of numerically trivial divisors. 
Consider the morphism
\begin{align*}
	\imath\colon C&\longrightarrow C_{n-1}\\
	c&\longmapsto c+(n-2)c_0
\end{align*}
which satisfies $u_{n-1}\circ\imath=u_1$.
The pullback $\imath^*\colon{\rm Pic}^0(C_{n-1})\to J(C)$ is an isomorphism with inverse $\alpha\mapsto T_{\alpha}$.
See \cite[\S3.2]{Sheridan}.
The class $u^*\kappa(\alpha)$
is a numerically trivial divisor $\eta$ on $C_{n-1}$ such that 
$\imath^*\eta=\alpha$. The claim follows.)

\texttt{Claim 3.} If $E$ is a divisor of degree $e$ on $C$ with $e\geq 2g-1$, then $R^iu_*\cal O_{C_{n-1}}(N_E)=0$ for $i>0$. (The pullback $u^*\theta$ has class $(g+n-2)x-\frac{\delta}2$, e.g., by \cite[Lemma 2.1]{Pacienza}. By the projection formula, it is enough to prove that if $F$ is a divisor of degree $f\geq -(n-1-g)$ on $C$, then $R^iu_*\cal O_{C_{n-1}}(T_F)=0$ for $i>0$. This is true because the fibers of $u$ are projective spaces $\bb P^{n-1-g}$, and $T_{F}$ restricts as $\cal O_{\bb P^{n-1-g}}(f)$ on them. These line bundles have no higher cohomology in the stated range.)

\texttt{Claim 4.} If $E$ is a divisor of degree $e$ on $C$ with $e\geq 2g-1$, then $h^i(C_{n-1},N_E)=0$ for $i>0$.(We have $K_{C_{n-1}}=N_{K_C}$. Then $N_{E}=N_{K_C}+T_{E-K_C}$. 
The claim follows by Kodaira vanishing.)
\medskip

Let us assume for the moment that the main claim is proved. 
By Lemma \ref{lem:transformSchur}.(\ref{lem:transformSchuri}), we have $p_*\cal O(q^*N_{L+\alpha_i}-Z)=\bigwedge^{n-1}M_{L+\alpha_i}$. This is semistable of slope $-\frac{(n-1)d}{d-g}$ since $d\geq g+n-1\geq 2g$ by \cite[Proposition 3.2]{elstable}, hence $p_*\cal O(q^*N_{L+\alpha_i}-Z)\bigl\langle\frac{(n-1)d}{d-g}c_0\bigr\rangle$ is nef by Lemma \ref{lem:bp14}.
From Claim 2, the line bundles $\cal O(q^*T_{\alpha_i^{\vee}})$ are numerically trivial. The twist of the LHS of \eqref{eq:continuousglobalgeneration} by $p^*\frac{(n-1)d}{d-g}c_0$ is then nef.
The conclusion follows as in the proof of Proposition \ref{prop:curveapprox}.(\ref{prop:curveapproxeasy}).
\medskip

\texttt{The proof of the main claim.}
The surjectivity of
\eqref{eq:continuousglobalgeneration} along the general fiber is implied by surjectivity on the fiber over $c_0$.
Note that $Z|_{\{c_0\}\times C_{n-1}}=c_0+C_{n-2}=T_{c_0}$.
By cohomology and base change (using Claim 4), the fiber map is
\begin{equation}\label{eq:continuousfiberrestriction}\bigoplus_{i=1}^NH^0\bigl(C_{n-1},N_{L-c_0+\alpha_i}\bigr)\otimes T_{\alpha_i^{\vee}}\longrightarrow\cal O_{C_{n-1}}(N_{L-c_{0}}).
	\end{equation}
The natural relative evaluation map $u^*u_*\cal O(N_{L-c_0})\to \cal O(N_{L-c_0})$ is surjective: arguing as in Claim 3, the divisor $L-c_0$ restricts as $\cal O(d-1-(g+n-2))$ on the fibers of $u$, and $d-1-(g+n-2)\geq 0$. Using Claim 2, the surjectivity of \eqref{eq:continuousfiberrestriction} follows after pullback from the surjectivity
of
\[\bigoplus_{i=1}^N H^0\bigl(J(C),u_*\cal O(N_{L-c_0+\alpha_i})\bigr)\otimes\kappa(\alpha_i^{\vee})\longrightarrow u_*\cal O(N_{L-c_0}).\]
From Claim 2 and from the projection formula, this is equivalent to the surjectivity of the map
\[\bigoplus_{i=1}^N H^0\bigl(J(C),u_*\cal O(N_{L-c_0})\otimes \kappa(\alpha_i)\bigr)\otimes\kappa(\alpha_i^{\vee})\longrightarrow u_*\cal O(N_{L-c_0}).\]
In this form, the result is a direct application of \cite[Corollary 2.4]{Pareschi} if we prove that 
\[H^i\bigl(J(C),u_*\cal O(N_{L-c_0})\otimes\kappa(\alpha)\bigr)=0\] for all $i>0$ and all $\alpha\in J(C)$. This is a consequence of Claims 3 and 4, and the Leray spectral sequence. 
\end{proof}

\begin{rmka}
	Using the three parts of Theorem \ref{thrm:prodcurvescgeneral}, the proof of Theorem \ref{thrm:mainCn}.(\ref{thrm:mainCni}) gives slightly better results when $C$ is a general curve.
\end{rmka}

\begin{lem}\label{lem:KMsteal}Let $C$ be a smooth projective curve of genus $g$, and let $0\leq n\leq g$. Let $D$ be a divisor of degree $d$ on $C$ such that $\cal O_C(D)$ is general in ${\rm Pic}^d(C)$. 
	If $d\geq 2n+g+1$, then $D$ is $n$-very ample.
\end{lem}
We follow an idea of \cite{KM20} where the case $n=0$ was treated.
\begin{proof}
	If $d\geq 2g+n$, then any $D$ is $n$-very ample by Riemann--Roch or Kodaira vanishing. Assume $d\leq 2g+n-1$ and $D$ is not $n$-very ample.
	Then there exists $Z_{n+1}$ an effective divisor of degree $n+1$ such that $h^1(C,D-Z_{n+1})\neq 0$, in particular $h^0(C,K_C-D+Z_{n+1})\neq 0$.
	Note that $\deg(K_C-D+Z_{n+1})=2g+n-1-d\geq 0$.
	We have $K_C-D+Z_{n+1}\sim E$, for some effective $E$ of degree $2g+n-1-d$. The pairs $(Z_{n+1},E)$ move in an $2g+2n-d$-dimensional family, and $2g+2n-d\leq g-1$. For a general choice of $\cal O_C(D)$
	we have $K_C-D\not\sim E-Z_{n+1}$ for any such pair $(Z_{n+1},E)$, which is a contradiction.
\end{proof}

\section{Other considerations}

\subsection{Product of very general distinct curves}
	If $C_1$ and $C_2$ are smooth projective curves, then ${\rm Pic}(C_1\times C_2)\simeq{\rm Pic}(C_1)\oplus{\rm Pic(C_2)}\oplus{\rm Hom}(J(C_1),J(C_2))$. If the curves are sufficiently general, there exists no nontrivial morphism between their Jacobians. In particular, $N^1(C_1\times C_2)$ is 2-dimensional, generated by the classes $f_1$ and $f_2$ of the fibers of each projection. In this case $\Nef(C_1\times C_2)=\Eff(C_1\times C_2)=\langle f_1,f_2\rangle$. 

\subsection{Restricting from larger ambient spaces}

\begin{rmk}[Restricting from the Jacobian]\label{rmk:delv}
	Let $C$ be an \emph{arbitrary} curve of genus $g\geq 1$.
	Let $J(C)$ be the Jacobian of $C$. 
	The ``canonical part'' of the nef cone of $J(C)\times J(C)$ is essentially known by \cite{delv11}. By restricting these classes to $C\times C$,
	one finds
	\[d\:\!f_1+\biggl(1+\frac{g^2}{d-1}\biggr)f_2-\delta\in\Nef(C\times C)\ (\forall)\ d>1.\]
	See Figure \ref{fig:nefconeArbitraryg10} for a comparison with Remark \ref{rmk:RabindranathVojta}.
\end{rmk}

\begin{rmka}[Restricting from Hilbert schemes of K3 surfaces]
	
	For $(S,H)$ a polarized K3 surface of degree $2t$ and Picard number 1,
	\cite{BM14} compute the Nef cone of ${\rm Hilb}^2S$ in terms of solutions to the Pell equations $x^2-4ty^2=5$ and $x^2-ty^2=1$.
	If $C\subset S$ is a smooth curve, one can pullback nef classes from ${\rm Hilb}^2S$ to $C\times C$. 
	The resulting examples are roughly of form $2\sqrt g(f_1+f_2)-\delta$,
	off by a factor of 2 from Conjecture \ref{conj:prodcurvesintro}.	
\end{rmka}

\subsection{Semistability of
\texorpdfstring{$M^{(m-1)}(L)$}{M\textasciicircum\{(m-1)\}(L)} via positivity on
\texorpdfstring{$C\times C_t$}{C \texttimes\ C\_t}}
	One can use Lemma \ref{lem:transformSchur}.(\ref{lem:transformSchuri}) to reduce the semistability of $M^{(m-1)}(L)$ to the non-effectivity of certain divisors on $C\times C_{t}$.
		If $C$ is a smooth projective curve over $\bb C$, and $\cal E$ is a locally free sheaf on $C$, then $\cal E$ is semistable if, and only if,
	it is cohomologically semistable in the sense that for all $1\leq t$ we have $h^0(C,A^{\vee}\otimes \bigwedge^t\cal E)=0$ for all line bundles $A$ of degree $a$ with $a>t\cdot\mu(\cal E)$. (The idea is that if $\cal F\subset\cal E$ destabilizes $\cal E$ and has rank $t$, then 
	we have an inclusion $A\coloneqq\det\cal F\subset\bigwedge^t\cal E$.)
	
	In particular, using Lemma \ref{lem:transformSchur}.(\ref{lem:transformSchuri}) and the projection formula, $M^{(m-1)}(L)$ is semistable if, and only if,
	$h^0(C\times C_{t}, -p^*A+q^*N_L-mZ_t)=0$ for all $t\geq 1$ and all divisors $A$ on $C$ with $\deg A>t\cdot\mu(M^{(m-1)}(L))$.
	
	Instead of proving the semistability of $M^{(m-1)}(L)$, one might be interested in bounding its positivity, which in characteristic zero is determined by $\mu_{\min}(M^{(m-1)}(L))$. 
	Here it is more natural to work with quotients instead of subsheaves. In terms of quotients, the semistability of $\cal E$ is determined by the vanishing of $h^0(C,B\otimes\bigwedge^t\cal E^{\vee})=h^1(C,\omega_C\otimes B^{\vee}\otimes\bigwedge^t\cal E)$ for all line bundle $B$ of degree $b$ with $b<t\cdot\mu(\cal E)$. 
	
	If $L$ is sufficiently positive so that $R^1p_*\cal O_{C\times C_t}(q^*N_L-mZ)=0$ for all $1\leq t<{\rm rk}\, M^{(m-1)}(L)$ (e.g., when $h^1(C_t,N_{L-mc})=0$ for all $c\in C$; this holds when $\deg L> 2g-2+m$), then we would be looking to understand $h^1(C\times C_t,p^*(K_C-B)+q^*N_L-mZ_t)$ for all divisors $B$ on $C$ of degree $b$ with $b<t\cdot\mu(M^{(m-1)}(L))$. 




\bibliographystyle{amsalpha}
\bibliography{ProductsOfCurves}

\providecommand{\bysame}{\leavevmode\hbox to3em{\hrulefill}\thinspace}
\providecommand{\MR}{\relax\ifhmode\unskip\space\fi MR }
\providecommand{\MRhref}[2]{%
  \href{http://www.ams.org/mathscinet-getitem?mr=#1}{#2}
}
\providecommand{\href}[2]{#2}
\begin{thebibliography}{ACGH85}

\bibitem[ACGH85]{acgh}
Enrico Arbarello, Maurizio Cornalba, Pillip~A. Griffiths, and Joe Harris,
  \emph{Geometry of algebraic curves. {V}ol. {I}}, Grundlehren Math. Wiss.,
  vol. 267, Springer-Verlag, New York, 1985. \MR{770932}

\bibitem[Bar71]{Barton71}
Charles~M. Barton, \emph{Tensor products of ample vector bundles in
  characteristic {$p$}}, Amer. J. Math. \textbf{93} (1971), 429--438.
  \MR{289525}

\bibitem[Bis05]{Biswas05}
Indranil Biswas, \emph{A criterion for ample vector bundles over a curve in
  positive characteristic}, Bull. Sci. Math. \textbf{129} (2005), no.~6,
  539--543. \MR{2142897}

\bibitem[BM14]{BM14}
Arend Bayer and Emanuele Macr\`\i, \emph{M{MP} for moduli of sheaves on {K}3s
  via wall-crossing: nef and movable cones, {L}agrangian fibrations}, Invent.
  Math. \textbf{198} (2014), no.~3, 505--590. \MR{3279532}

\bibitem[BP14]{BP14}
Indranil Biswas and A.~J. Parameswaran, \emph{Nef cone of flag bundles over a
  curve}, Kyoto J. Math. \textbf{54} (2014), no.~2, 353--366. \MR{3215571}

\bibitem[BPO09]{BrambilaOrtega}
L.~Brambila-Paz and Angela Ortega, \emph{Brill-{N}oether bundles and coherent
  systems on special curves}, Moduli spaces and vector bundles, London Math.
  Soc. Lecture Note Ser., vol. 359, Cambridge Univ. Press, Cambridge, 2009,
  pp.~456--472. \MR{2537078}

\bibitem[Bre04]{Brenner04}
Holger Brenner, \emph{Slopes of vector bundles on projective curves and
  applications to tight closure problems}, Trans. Amer. Math. Soc. \textbf{356}
  (2004), no.~1, 371--392. \MR{2020037}

\bibitem[Bre06]{Brenner06}
\bysame, \emph{Tight closure and plus closure in dimension two}, Amer. J. Math.
  \textbf{128} (2006), no.~2, 531--539. \MR{2214902}

\bibitem[But94]{Butler94}
David~C. Butler, \emph{Normal generation of vector bundles over a curve}, J.
  Differential Geom. \textbf{39} (1994), no.~1, 1--34. \MR{1258911}

\bibitem[Cam08]{Camere}
Chiara Camere, \emph{About the stability of the tangent bundle restricted to a
  curve}, C. R. Math. Acad. Sci. Paris \textbf{346} (2008), no.~7-8, 421--426.
  \MR{2417562}

\bibitem[CG90]{CG90}
Fabrizio Catanese and Lothar G{\oe}ttsche, \emph{{$d$}-very-ample line bundles
  and embeddings of {H}ilbert schemes of {$0$}-cycles}, Manuscripta Math.
  \textbf{68} (1990), no.~3, 337--341. \MR{1065935}

\bibitem[CK99]{cknagata}
Ciro Ciliberto and Alexis Kouvidakis, \emph{On the symmetric product of a curve
  with general moduli}, Geom. Dedicata \textbf{78} (1999), no.~3, 327--343.
  \MR{1725369}

\bibitem[dBE52]{deBruijnErdos}
N.~G. de~Bruijn and P.~Erd\H{o}s, \emph{Some linear and some quadratic
  recursion formulas. {II}}, Nederl. Akad. Wetensch. Proc. Ser. A. {\bf 55} =
  Indagationes Math. \textbf{14} (1952), 152--163. \MR{47162}

\bibitem[DELV11]{delv11}
Olivier Debarre, Lawrence Ein, Robert Lazarsfeld, and Claire Voisin,
  \emph{Pseudoeffective and nef classes on abelian varieties}, Compos. Math.
  \textbf{147} (2011), no.~6, 1793--1818. \MR{2862063}

\bibitem[DM69]{DM69}
P.~Deligne and D.~Mumford, \emph{The irreducibility of the space of curves of
  given genus}, Inst. Hautes \'{E}tudes Sci. Publ. Math. (1969), no.~36,
  75--109. \MR{262240}

\bibitem[EL92]{elstable}
Lawrence Ein and Robert Lazarsfeld, \emph{Stability and restrictions of
  {P}icard bundles, with an application to the normal bundles of elliptic
  curves}, Complex projective geometry ({T}rieste, 1989/{B}ergen, 1989), London
  Math. Soc. Lecture Note Ser., vol. 179, Cambridge Univ. Press, Cambridge,
  1992, pp.~149--156. \MR{1201380}

\bibitem[EL15]{ELgonality}
\bysame, \emph{The gonality conjecture on syzygies of algebraic curves of large
  degree}, Publ. Math. Inst. Hautes \'{E}tudes Sci. \textbf{122} (2015),
  301--313. \MR{3415069}

\bibitem[EN18]{EN18}
Lawrence Ein and Wenbo Niu, \emph{Interpolation for curves of large degree},
  Asian J. Math. \textbf{22} (2018), no.~2, 307--316. \MR{3824570}

\bibitem[ES12]{EusenSchreyer}
Friedrich Eusen and Frank-Olaf Schreyer, \emph{A remark on a conjecture of
  {P}aranjape and {R}amanan}, Geometry and arithmetic, EMS Ser. Congr. Rep.,
  Eur. Math. Soc., Z\"{u}rich, 2012, pp.~113--123. \MR{2987656}

\bibitem[FM21]{fm19}
Mihai Fulger and Takumi Murayama, \emph{Seshadri constants for vector bundles},
  J. Pure Appl. Algebra \textbf{225} (2021), no.~4, 106559, 35. \MR{4158762}

\bibitem[Ful11]{ful11}
Mihai Fulger, \emph{Cones of effective cycles on projective bundles over
  curves}, Math. Z. \textbf{269} (2011), no.~1-2, 449--459. \MR{2836078}

\bibitem[Har71]{har71}
Robin Hartshorne, \emph{Ample vector bundles on curves}, Nagoya Math. J.
  \textbf{43} (1971), 73--89. \MR{292847}

\bibitem[KM20]{KM20}
John Kopper and Sayanta Mandal, \emph{Non-globally generated bundles on
  curves}, 2020, arXiv:2003.10890 [math.AG].

\bibitem[Kou93]{kouvidakis}
Alexis Kouvidakis, \emph{Divisors on symmetric products of curves}, Trans.
  Amer. Math. Soc. \textbf{337} (1993), no.~1, 117--128. \MR{1149124}

\bibitem[Lan04]{Langer04}
Adrian Langer, \emph{Semistable sheaves in positive characteristic}, Ann. of
  Math. (2) \textbf{159} (2004), no.~1, 251--276. \MR{2051393}

\bibitem[Laz04a]{laz04}
Robert Lazarsfeld, \emph{Positivity in algebraic geometry. {I}}, Ergeb. Math.
  Grenzgeb. (3), vol.~48, Springer-Verlag, Berlin, 2004, Classical setting:
  line bundles and linear series. \MR{2095472}

\bibitem[Laz04b]{laz042}
\bysame, \emph{Positivity in algebraic geometry. {II}}, Ergeb. Math. Grenzgeb.
  (3), vol.~49, Springer-Verlag, Berlin, 2004, Positivity for vector bundles,
  and multiplier ideals. \MR{2095472}

\bibitem[Mis19]{MistrettaStability}
Ernesto~C. Mistretta, \emph{On stability of tautological bundles and their
  total transforms}, Milan J. Math. \textbf{87} (2019), no.~2, 273--282.
  \MR{4031784}

\bibitem[Miy87]{Miy87}
Yoichi Miyaoka, \emph{The {C}hern classes and {K}odaira dimension of a minimal
  variety}, Algebraic geometry, {S}endai, 1985, Adv. Stud. Pure Math., vol.~10,
  North-Holland, Amsterdam, 1987, pp.~449--476. \MR{946247}

\bibitem[MS12]{MisStop}
Ernesto~C. Mistretta and Lidia Stoppino, \emph{Linear series on curves:
  stability and {C}lifford index}, Internat. J. Math. \textbf{23} (2012),
  no.~12, 1250121, 25. \MR{3019424}

\bibitem[Nag59]{NagataConj}
Masayoshi Nagata, \emph{On the {$14$}-th problem of {H}ilbert}, Amer. J. Math.
  \textbf{81} (1959), 766--772. \MR{105409}

\bibitem[Pac03]{Pacienza}
Gianluca Pacienza, \emph{On the nef cone of symmetric products of a generic
  curve}, Amer. J. Math. \textbf{125} (2003), no.~5, 1117--1135. \MR{2004430}

\bibitem[Par00]{Pareschi}
Giuseppe Pareschi, \emph{Syzygies of abelian varieties}, J. Amer. Math. Soc.
  \textbf{13} (2000), no.~3, 651--664. \MR{1758758}

\bibitem[PR88]{PR88}
Kapil Paranjape and S.~Ramanan, \emph{On the canonical ring of a curve},
  Algebraic geometry and commutative algebra, {V}ol. {II}, Kinokuniya, Tokyo,
  1988, pp.~503--516. \MR{977775}

\bibitem[Rab19]{ashwath}
Ashwath Rabindranath, \emph{Some surfaces with non-polyhedral nef cones}, Proc.
  Amer. Math. Soc. \textbf{147} (2019), no.~1, 15--20. \MR{3876727}

\bibitem[Ros07]{Ross}
Julius Ross, \emph{Seshadri constants on symmetric products of curves}, Math.
  Res. Lett. \textbf{14} (2007), no.~1, 63--75. \MR{2289620}

\bibitem[She19]{Sheridan}
John Sheridan, \emph{Divisor varieties of symmetric products}, 2019,
  arXiv:1906.05465[math.AG].

\bibitem[SSS04]{SSS}
Beata Strycharz-Szemberg and Tomasz Szemberg, \emph{Remarks on the {N}agata
  conjecture}, Serdica Math. J. \textbf{30} (2004), no.~2-3, 405--430.
  \MR{2098342}

\bibitem[Voj89]{vojta}
Paul Vojta, \emph{Mordell's conjecture over function fields}, Invent. Math.
  \textbf{98} (1989), no.~1, 115--138. \MR{1010158}

\bibitem[Zha17]{Zhao17}
Yifei Zhao, \emph{Maximally {F}robenius-destabilized vector bundles over smooth
  algebraic curves}, Internat. J. Math. \textbf{28} (2017), no.~2, 1750003, 26.
  \MR{3615581}

\end{thebibliography}

\end{document}